\documentclass[11pt]{amsart}

\usepackage{todonotes}
\usepackage{dsfont}  % to write indicator function
\usepackage{caption, subcaption, graphicx} % to set width in figures
\usetikzlibrary{calc} % to draw the exp function 
\usepackage{amsmath}  % to draw the exp function
\usepackage{mathtools}
\mathtoolsset{showonlyrefs}

\usepackage{epstopdf} %converting to PDF

\usepackage{xcolor} % to define light grey
\definecolor{mygray}{gray}{0.75} %ight grey

\usepackage[normalem]{ulem} % to strike out text

\usepackage{geometry}
\newtheorem{definition}{Definition}[section]
\newtheorem{theorem}[definition]{Theorem}

\newtheorem{proposition}[definition]{Proposition}
\newtheorem{lemma}[definition]{Lemma}
\newtheorem{remark}[definition]{Remark}

\newtheorem{assumption}{Assumption}
\newtheorem*{setting*}{Setting}

\newcommand{\norme}[1]{\left\Vert #1\right\Vert}
\newcommand{\R}{\mathbb{R}}
\newcommand{\E}{\mathbb{E}}
\newcommand{\N}{\mathbb{N}}
\newcommand{\Z}{\mathbb{Z}}
\newcommand{\di}{\mathrm{d}}
\newcommand{\eps}{\varepsilon}

\def\XX{\widehat{X}}
\def\YY{\widehat Y}

\def\aa{\hat a}

\def\C{\mathcal{C}}

\def\P{\mathbb{P}}

\begin{document}
\title[Numerical scheme for SDEs with distributional drift]{A numerical scheme for stochastic differential equations with distributional drift}
\author[T.~De Angelis, M.~Germain, E.~Issoglio]{Tiziano De Angelis, Maximilien Germain, Elena Issoglio}
\address{T.~De Angelis: School of Management and Economics, Dept. ESOMAS, University of Turin, C.so
Unione Sovietica 218bis, 10134, Torino, Italy. Collegio Carlo Alberto, P.za Arbarello 8, 10122, Torino, Italy}\email{tiziano.deangelis@unito.it}
\address{M.~Germain: LPSM, Universit\'e de Paris\newline B\^atiment Sophie Germain, Case 7012, 8 place Aur\'elie Nemours,
	75205 Paris Cedex 13, France}\email{mgermain@lpsm.paris}
\address{E.~Issoglio: Dept.\ of Mathematics ``G.\ Peano'', University of Turin, Via Carlo Alberto 10, 10123, Torino, Italy.}\email{elena.issoglio@unito.it}
\date{\today}

\begin{abstract}
In this paper we introduce a scheme for the numerical solution of one-dimensional stochastic differential equations (SDEs) whose drift belongs to a fractional Sobolev space of negative regularity (a subspace of Schwartz distributions). We obtain a convergence rate in a suitable $L^1$-norm and, as a by-product, a convergence rate for a numerical scheme applied to SDEs with drift in $L^p$-spaces with $p\in(1,\infty)$. 
\end{abstract}


\thanks{{\em Acknowledgments}: T.~De Angelis gratefully acknowledges support by the EPSRC grant EP/R021201/1. Parts of this work were carried out while E.\ Issoglio and T.\ De Angelis were employed by the University of Leeds. M.\ Germain's visited the University of Leeds from May to August 2018. We thank the School of Mathematics at the University of Leeds for the hospitality. M.~Germain gratefully acknowledges support from  Erasmus+ and Universit\'e Paris-Saclay.}

\maketitle

\section{Introduction} 
The aim of our paper is to obtain a numerical algorithm (and its convergence rate) capable of approximating the solution of a one-dimensional SDE of the form
\begin{equation} \label{eq: original sde}
\di X_t = b(t,X_t)\di t + \di W_t,\qquad X_0 = x,\qquad t\in[0,T],
\end{equation}
where $W$ is a Brownian motion on a probability space $(\Omega,\mathcal F,\P)$ and $b$ is a distributional drift. In particular, $b(t)$ takes values in a fractional Sobolev space of negative order (defined in Section \ref{sec:setting}) for each $t\in [0,T]$, and $t\mapsto b(t)$ is H\"older continuous (i.e., $b\in \mathcal C^{\kappa}([0,T]; H^{-\beta_0}_{\tilde q_0, q_0} )$ for $\kappa\in(1/2,1)$, $\beta_0\in(0,1/4)$ and suitable $\tilde q_0$ and $q_0$).
Existence and uniqueness of solutions for $d$-dimensional versions of \eqref{eq: original sde} were first derived in Flandoli et al.\ \cite{Fla-Iss-Rus-2017}, where the authors give a mathematical meaning to the term $\int_0^t b(s, X_s) \di s$ by introducing the concept of \emph{virtual solution}. The latter is needed since pointwise evaluation of $b(t,\,\cdot\,)$ is meaningless. Further theoretical work on equations of a similar kind can be found, for example, in Cannizzaro and Chouk \cite{cannizzaro2018}, Delarue and Diel \cite{Delarue2016}, Flandoli et al.\ \cite{FlandoliRussoWolf}, Issoglio and Jing \cite{issoglio2019}, Issoglio and Russo \cite{Iss-Rus-2200}. Besides their theoretical interest in the context of regularisation by noise, these singular SDEs usually provide models for random irregular media. For example, Russo and Trutnau \cite{RussoTrutnau} use them in the study of singular Stochastic Partial Differential Equations (SPDEs). Other works such as \cite{cannizzaro2018,Delarue2016} explain how certain SDEs with distributional drift 
can be used to describe the infinitesimal behaviour of the so-called polymer measure (a singular measure on the space of continuous functions) with links to the celebrated KPZ equation. 
Another example of SDE with distributional drift is contained in Hu et al.\ \cite{BroxLe} which studies the so-called Brox diffusion, introduced by Brox \cite{Brox} as an example of random process in a random medium. In the framework of \cite{BroxLe} the drift has the regularity of the distributional derivative of a two-sided Brownian motion.

For practical implementation, mathematical models as the ones mentioned above require numerical schemes designed to handle distributional drifts. Our work provides results and tools in that direction. 
We devise a 2-step algorithm for the numerical solution of \eqref{eq: original sde}:  we first regularise the drift and then apply Euler-Maruyama scheme. This produces a sequence $(X^N)_{N\ge 1}$ of solutions of SDEs with a smooth drift and a sequence $(X^{N,m})_{m\ge 1}$ of corresponding Euler-Maruyama approximations for each $N\ge 1$. 
We prove in Theorem \ref{thm:global} that the scheme converges to the original virtual solution of \eqref{eq: original sde} and obtain a strong $L^1$-rate of convergence when we let $(N,m)\to\infty$ simultaneously, i.e., we obtain a rate of convergence for the limit
\[
\lim_{(N,m)\to\infty}\sup_{0\le t\le T}\E\left[\big|X_t-X^{N,m}_t\big|\right].
\]

In order to regularise the drift we first pick a sequence of functions $(b^N)_{N\ge 1}$ that converges to $b$ in the appropriate norm for the fractional Sobolev space of negative regularity (i.e., it converges in $\mathcal C^{\kappa}([0,T]; H^{-\beta_0}_{\tilde q_0, q_0} )$). Then, we further mollify the functions $b^N$ by convolution with the heat kernel (we refer to it as {\em randomisation} procedure). 
The first step removes the difficulty of working with distributions, while 
the randomisation allows us to control the convergence rate of the overall scheme as $(N,m)\to\infty$ simultaneously.
As explained in detail at the beginning of Section \ref{sec:thm2}, this approach gives us a better convergence rate than the one we would obtain if we omitted the randomisation procedure and relied upon known (tight) bounds in Euler-Maruyama schemes for SDEs with bounded measurable drift.  

Due to the distributional nature of our drift, the actual numerical implementation of the scheme is non-trivial and, in particular, the choice of the approximating functions $b^N$, in the first step of the algorithm, needs to be addressed carefully. In general, an explicit representation of $b$ can be provided in terms of an infinite sum of Haar wavelets (see Appendix \ref{sc: haar}) and we construct the sequence $(b^N)_{N\ge 1}$ by suitably truncating such series representation. The main reasons for this choice are: (i) these wavelets form a basis for the fractional Sobolev spaces of negative order which are needed to accommodate the original drift $b$; (ii) they enjoy the so-called {\em multi-resolution property}, which improves the computational efficiency of the algorithm; (iii) since Haar wavelets are piecewise constant functions, their convolution with the heat kernel only requires knowledge of the cumulative Gaussian distribution, hence requiring no additional computational effort for our randomisation procedure. Crucially, thanks to property (i) above, we are able to determine the convergence rate of $b^N$ to $b$ (see Proposition \ref{prop:Haar}).

A special case of distributional drift is obtained when $b$ is of Dirac-delta type. That leads to one-dimensional SDEs with local-time and the celebrated skew Brownian motion (see Harrison and Shepp \cite{harrison1981skew}; see also \'Etor\'e and Martinez \cite{EMskew} for the time-inhomogeneous case). Such SDEs have been widely studied in the literature, including several works on numerical approximations (see, e.g., \'Etor\'e and Martinez \cite{EMexact} and various contributions by Lejay and co-authors \cite{lejay,lejayetal,LM06}). Properties of the transition density and resolvent of the process, together with links to It\^o and McKean's theory of one-dimensional diffusions (\cite{itomckean}), enable efficient numerical methods. Those methods cannot be applied in our setting, where the process $X$ solution of \eqref{eq: original sde} is not necessarily a semi-martingale, as shown in Flandoli et al.\ \cite[Cor.\ 5.11]{FRW04}.

Except for the case of skew diffusions, our work seems to be the first to address numerical methods for a class of SDEs whose drift is merely a distribution.  This advancement on all the existing results hinges on the concept of virtual solution given by Flandoli et al. \cite{Fla-Iss-Rus-2017}, which links the SDE in \eqref{eq: original sde} to a class of partial differential equations (PDEs) with distributional drift studied in Issoglio \cite{Elena}. 
It is worth emphasising that our algorithm {\em does not} require a numerical solution of the PDE and instead it deals directly with the SDE in \eqref{eq: original sde}. Hence, the methods that we use here can be adopted to complement/extend the existing studies on numerical schemes for SDEs whose drift is a {\em function} with low regularity.

The literature on Euler-Maruyama approximation of SDEs whose drift is some function with low regularity is very vast and here we only provide a short overview. Early contributions are due to Gy\"ongy and Krylov \cite{Gyongy-Krylov-1996} who obtain convergence in probability for SDEs with continuous coefficients. A few years later Yan \cite{Yan} proves weak convergence for SDEs with particular discontinuities of the coefficients and an $L^1$-rate of convergence under the assumption of a Lipschitz-continuous drift and a H\"older-continuous diffusion coefficient, in a one-dimensional setting.
Also Halidias and Kloeden \cite{Hal-Klo} prove strong convergence in $L^2$ (but with no rate) when the coefficients allow certain types of discontinuity. 

More recent results include work by M\"uller-Gronbach and Yaroslavtseva \cite{muller2018}, who obtain an $L^p$-rate of $1/2$ (for any $p\ge 1$) for one-dimensional SDEs with discontinuous drift, and work by Neuenkirch et al.\ \cite{neuenkirch2019} where analogous results are obtained in a multi-dimensional setting with respect to an $L^2$-norm. Neuenkirch and Sz\"olgyenyi \cite{neuenkirch2019b} instead find an $L^2$-rate of up to $3/4$ for one-dimensional diffusions with possibly discontinuous drift (with Sobolev-Slobodeckij type regularity). Further related results can also be found in Leobacher and Sz\"olgyenyi \cite{leobacher2017} where an $L^2$-rate of convergence of $1/2$ is obtained for (possibly degenerate) multi-dimensional SDEs. Notice that in \cite{leobacher2017} the Euler-Maruyama scheme is applied to a process obtained as a suitable transformation of the solution of the SDE. Similar ideas were also used in another paper by the same authors (\cite{leobacher_et_al18}) to find an $L^2$-rate of convergence of $1/4$ but, differently from \cite{leobacher2017}, the convergence in \cite{leobacher_et_al18} is for the approximation of the original SDE. Numerical schemes for non-degenerate SDEs with irregular coefficients are also addressed in works by Ngo and Taguchi \cite{Ngo-Tag} (multi-dimensional setting, rate $1/4$) and \cite{ngo2017} (one-dimensional setting, rate $1/2$).

Our approach is close in spirit to the one adopted by Dareiotis and Gerencs\'er \cite{dareiotis2020}, who use the regularising effect of the Brownian noise to obtain a strong convergence rate of $1/2$ for multidimensional SDEs with continuous drift and, in the one-dimensional case, for SDEs with bounded drift. We discuss extensively differences between their approach and ours at the beginning of Section \ref{sec:thm2}.
Prior to their work, Menoukeu-Pamen and Taguchi \cite{pamen} had obtained strong rate of convergence in $L^p$ of order $ p\beta/2$ for $d$-dimensional SDEs with $\beta$-H\"older continuous coefficients.
 Finally, we would like also to mention a new approach developed by Butkovsky et al.\ \cite{butkovsky2019}, who use regularisation by noise and a so-called stochastic sewing lemma to obtain convergence rates for SDEs driven by fractional Brownian motion and a convergence rate of $1/2$ for SDEs with continuous drift and multiplicative Brownian noise.

There are also numerous results on weak convergence of Euler-Maruyama approximation of SDEs, however a detailed review falls outside the scopes of our paper. For example, when coefficients are smooth, convergence with rate up to $1$ was obtained by Bally and Talay \cite{bally1996law} (also work by Mikulevicius and Platen \cite{mikulevicius1991} contains further results in that direction). In the case of irregular drift, a scheme in two steps is analysed by Kohatsu-Higa et al. \cite{Koh-Lej-Kaz}. They first regularise the drift of their SDE and then apply Euler-Maruyama scheme to the more regular process and obtain a rate of weak convergence.

We note that a direct comparison of the rate we obtain in the case of distributional drift and the rates obtained in the various papers mentioned above is not necessarily meaningful: the methods used in those papers require that the drift be a function and do not allow easy extensions to the distributional case. 
Finally,  it is worth noticing that our results also apply to SDEs with drift in $L^p$-spaces with $p\in(1,\infty)$ (Remark \ref{rem:Lp}) for which no rate is known yet.

The paper is organised as follows. In Section \ref{sec:setting} we introduce the necessary notation, including the fractional Sobolev spaces of negative order that our drift belongs to. Then in Section \ref{sec: scheme} we introduce the numerical scheme. In Section \ref{sc: main} we present the main results of the paper, whose proofs are then provided in Sections \ref{sec:thm1} and \ref{sec:thm2}. Background material on SDEs with distributional drift, which is needed to understand our arguments of proof, is presented in Section \ref{sec:background}. 
The paper is completed by a technical appendix that accounts for important properties of Haar wavelets and a short appendix on standard estimates for the (killed) heat semigroup.

\section{Setting and numerical scheme}
\subsection{Notation}\label{sec:setting}

Here we introduce the functional analytic framework needed for the well-posedness of equation \eqref{eq: original sde}.  Throughout the paper we will use $\nabla$ and $\Delta$ for the spatial gradient and Laplacian of a function, respectively, and $\partial_t$ for its partial derivative with respect to time.

For any Banach space $(B, \|\cdot\|_B)$ we denote by  $\mathcal{C}\left([0,T];B\right)$ the space of $B$-valued continuous functions of time. This is again a Banach space when endowed with the norm $\norme{f}_{\infty,B} =  \sup_{0\leq t \leq T} \norme{f(t)}_{B}$. For future reference we also introduce on $\mathcal{C}\left([0,T];B\right)$ the family of equivalent norms
\begin{align}\label{eq:norm-rho}
\norme{f}^{(\rho)}_{\infty,B} :=  \sup_{0\leq t \leq T} e^{-\rho t}\norme{f(t)}_{B},\quad\text{for $\rho\ge 0$}.
\end{align}
For $\alpha\in(0,1)$ we   introduce the subspace $\mathcal{C}^\alpha([0,T];B)$ of functions $f\in \mathcal{C}([0,T];B)$ such that 
\begin{align}\label{eq:Hsemi}
[f]_{\alpha,B}:=\sup_{s\neq t\in[0,T]}\frac{\|f(t)-f(s)\|_B}{|t-s|^\alpha}<\infty.
\end{align}
This is also a Banach space when endowed with the norm $\norme{f}_{\alpha,B} =  \norme{f}_{\infty,B}+[f]_{\alpha,B}$.

For $1\le r\le \infty$ we have the usual $L^r(\R)$-spaces endowed with the norms $\|\cdot\|_{L^r}$ and we use the short-hand notation $L^r$. We denote by  $\mathcal{C}^{0,0}(\R)$ and $\mathcal{C}^{1,0}(\R)$ the closure of the space $\mathcal{S}(\R)$ of Schwartz functions with respect to the norms  $ \norme{f}_{L^\infty} $ and  $\norme{f}_{L^\infty} + \norme{\nabla f}_{L^\infty}$, respectively. For simplicity of notation we just write ${\mathcal{C}^{0,0}}$ and ${\mathcal{C}^{1,0}}$. 
Further, we define the space of continuous functions (respectively, continuously differentiable functions) which are $\alpha$-H\"older continuous (respectively,  with $\alpha$-H\"older continuous first derivatives) for $0<\alpha<1$, that is the spaces
\begin{align*}
&\mathcal{C}^{0,\alpha}(\R):=\{f\in\mathcal{C}^{0,0}(\R)\ \text{ such that } \norme{f}_{\mathcal{C}^{0,\alpha}} <\infty \},\\
&\mathcal{C}^{1,\alpha}(\R):=\{f\in\mathcal{C}^{1,0}(\R)\ \text{ such that } \norme{f}_{\mathcal{C}^{1,\alpha}} <\infty \},
\end{align*} 
where the norms are defined as 
\begin{align*}
&\norme{f}_{\mathcal{C}^{0,\alpha}} := \norme{f}_{L^\infty} + \underset{x\neq y \in\R}{\sup} \frac{\left| f(x) - f(y)\right|}{|x-y|^\alpha},\\
&\norme{f}_{\mathcal{C}^{1,\alpha}} := \norme{f}_{L^\infty} + \norme{\nabla f}_{L^\infty} + \underset{x\neq y \in\R}{\sup} \frac{\left|\nabla f(x) - \nabla f(y)\right|}{|x-y|^\alpha}.
\end{align*}
For simplicity of notation we write ${\mathcal{C}^{0,\alpha}}$ and $\mathcal{C}^{1,\alpha}$ instead of $\mathcal{C}^{0,\alpha}(\R)$  and ${\mathcal{C}^{1,\alpha}(\R)}$, respectively.

For all $s\in \mathbb R$ and $r\in(1,\infty)$ we denote by $H^s_r(\mathbb R)$ the fractional Sobolev spaces (or Bessel-potential spaces) defined as the image of $L^r(\mathbb R)$ through  fractional powers of $A:=I-\frac12 \Delta$, i.e.,  $H^s_r(\mathbb R): = A^{-s/2}(L^r(\mathbb R))$ (for more details on fractional powers of $A$ see \cite[Remark 1.2]{Tri-bas}). This representation corresponds to 
\begin{align*}
    H^s_r(\mathbb R): = \Big\{f\in \mathcal{S}'(\R)\ |\ \mathfrak{F}^{-1} \Big( (1+|\xi|^2)^{\frac{s}{2}} \mathfrak{F}f\Big) \in L^r \Big \},
\end{align*} 
where $\mathfrak{F}$ is the Fourier transform and $\mathcal S'$ the space of Schwartz distributions (\cite[Sec.\ 2.2.2, Eq.\ (11)]{Tri-fun1}). For instance, $H^s_r(\mathbb R)$ with $s=\alpha-1$ contains the distributional derivative of $\alpha$-Hölder-continuous functions with compact support (see, e.g., \cite[Prop.\ 4.1]{Elena}). 
These spaces are Banach spaces when equipped with the norm 
\[
\|f\|_{H^s_r(\mathbb R)} := \| A^{s/2} f\|_{L^r(\mathbb R)},
\]
and $A^{\nu/2}$ is an isomorphism from $H^s_r(\R)$ to $H^{s-\nu}_r(\R)$ for all $\nu\in\R$, see  again \cite{Tri-bas}.

We observe that if $s<0$ then $H^s_r(\R)$ does actually contain distributions, while when $s\geq 0$ it only contains (measurable) functions. For $s=0$ we have the special case $H^0_r(\R) = L^r(\R)$. These spaces enjoy the following inclusion property: for  $ 1<r\leq u<\infty $ and $-\infty< t\leq s <\infty$ such that  $s-\frac1r \geq t-\frac1u$ we have 
\begin{equation}\label{eq:emb}
H^s_r (\R) \subset H^t_u (\R)
\end{equation}
(see \cite[Theorem 2.8.1]{Tri-fun}) and, in particular,  $H^s_r(\R)\subset H^{s-\nu}_r(\R)$ for $\nu>0$. 
Setting $H_{p,q}^s(\mathbb R) := H^s_p(\mathbb R)\cap H^s_q(\mathbb R)$, by interpolation we have that if $f\in H_{p,q}^s(\mathbb R)$ then $f$ is an element of all spaces $H_{r}^s(\mathbb R)$ for $p\wedge q <r< p \vee q$.  For simplicity we will use $H^s_r$ for the space $H^s_r(\R)$ and, analogously, we denote the associated norm by $\|\cdot\|_{H^s_r}$.

We denote by $(P_t)_{t\ge 0}$ the (killed) heat semigroup on $L^r $ generated by $-A$, that is, the semigroup with kernel $p(t,x)= e^{-t} ({2\pi t})^{-1/2} \exp \{-\tfrac{x^2}{2t}\} $. This is a bounded analytic semigroup and $D(A^{s/2}) = H^s_r$ for $s\in\R$, see \cite{Iss-Rus-2200, triebel78} for details.  
Moreover for all $t>0$ and all $s\in\R$ the operator $P_t$ maps $H^s_r$ into itself and, furthermore,  for any $f\in H^{-s}_r$  and for any $\varepsilon>0$  there is a constant $c=c_{\eps}>0$ such that 
\begin{equation}\label{eq:IZ1}
\begin{split}
&\|P_{t}\,f\|_{H^{-s}_{r}}\le \|f\|_{H^{-s}_{r}}\,,\\
&\|P_{t}\,f-f\|_{H^{-(s+\eps)}_{r}}\le c\, t^{\eps/2}\|f\|_{H^{-s}_{r}}\,.
\end{split}
\end{equation}
These bounds can be obtained (as in, e.g., \cite{issoglio2015}) using the following facts: (i) $P_t$ is a contraction on $L^r$, (ii) for $s>0$ the operators $A^{s/2}$ and $P_t$ commute by \cite[Theorem II.6.13]{pazy83} and (iii) $A^{s/2} $ is an isomorphism as mentioned above. For more details, see Appendix \ref{sc:P}.
 
We will also need estimates for the $L^\infty$-norm of $P_tf $ and of its gradient. 
To get those estimates, we use the fractional Morrey inequality  (\cite[Theorem 2.8.1, Remark 2]{triebel78}) which guarantees the embeddings 
\begin{equation}\label{eq:morrey}
H_r^{\nu}(\R)\subset\mathcal{C}^{0,\alpha}(\R) \quad \text{and} \quad H_r^{1+\nu}(\R)\subset\mathcal{C}^{1,\alpha}(\R),
\end{equation} 
if $\alpha:=\nu-1/r>0$. Then, using arguments similar to those for \eqref{eq:IZ1}, combined with fractional Morrey inequality, one 
obtains 
that for all $f\in H^{-s}_r$, $t\in(0, T]$ and $\nu>1/r$ we have
\begin{equation}\label{eq:IZ2}
\begin{split}
&\|P_{t}\,f\|_{L^\infty}\le c\, t^{-(s+\nu)/2}\|f\|_{H^{-s}_{r}}\,,\\
&\|\nabla(P_{t}\,f)\|_{L^\infty}\le c\, t^{-(1 +s+\nu)/2}\|f\|_{H^{-s}_{r}}\,,
\end{split}
\end{equation}
where $c\!>\!0$ varies from line to line and depends on $ T$. 
For more details, see Appendix \ref{sc:P}.

%%%%%%%%%%%%%%%%%%%%%%%%%%%%%%%%%%%%%%%%%%%%%%%%%%%%%%%%%%%%%%%%%%%
\subsection{Description of the scheme}\label{sec: scheme}
Our numerical scheme for \eqref{eq: original sde} is based on two subsequent approximations and a randomisation procedure. In order to justify pointwise evaluation of the distributional coefficient $b$, we approximate it by a sequence of bounded functions $(b^N)_{N\ge 1}$ that converges to $b$ in a suitable norm (see Assumption \ref{ass: b^N}). 
We further mollify the sequence $(b^N)_{N\ge 1}$ by convolution with the (killed) heat kernel and then we apply a generalised Euler-Maruyama scheme. The mollification can be interpreted as a randomisation procedure in space and it allows us to obtain a uniform rate of convergence for the overall scheme (see the discussion at the beginning of Section \ref{sec:thm2}).

To fix notation, let us consider a bounded measurable function $b^N:[0,T]\times\R\to \R$ and fix a constant $\eta_N>0$. Then the SDE
\begin{equation} \label{eq: approximated sde}
\di X^N_t = \big(P_{\eta_N}\, b^N\big)\left(t,X^N_t\right) \di t + \di W_t,\qquad X_0^N = x,
\end{equation} 
admits a unique strong solution. Note that in \eqref{eq: approximated sde} we slightly abuse the notation, because the solution $X^N$ depends both on $N$ and $\eta_N$ but we only indicate the dependence on $N$. Here and in what follows we always consider $\eta_N\to 0$ as $N\to\infty$.

Let $m\in \mathbb N$ and let us consider an equally-spaced partition of $[0,T]$ by setting $t_k:=\frac{kT}m$ for $k=0,1,\ldots, m$. Further, let us define
\[
k(t):=\sup\{k:t_k\le t\}.
\] 
Then the Euler-Maruyama approximation of the solution $X^N$ is given by 
\begin{align}\label{eq:EuM}
X^{N,m}_t=x+\int_0^t \big(P_{\eta_N}\,b^N\big)\left(t_{k(s)},X^{N, m}_{t_{k(s)}}\right)\di s+W_t,
\end{align}
and it is computed numerically according to    
\begin{equation}\label{eq: scheme}
    X^{N,m}_t=X^{N,m}_{t_{k(t)}}+\big(P_{\eta_N}\, b^N\big)\left(t_{k(t)},X^{N, m}_{t_{k(t)}}\right)(t-t_{k(t)})+\sqrt{t-t_{k(t)}}\,\, \varepsilon_k,
\end{equation}
with $(\varepsilon_k)_{k}$ i.i.d.\ standard Gaussian random variables.  

In general, the numerical implementation of the scheme  is more complicated than a standard Euler scheme since the mollified drift $P_{\eta_N}\, b^N$ may not be easily computable. However, we choose $b^N$ as a finite linear combination (with time-dependent coefficients) of Haar wavelets, which are piecewise constant functions with a very simple structure. Haar wavelets are convenient because they form an unconditional Schauder basis for the space $H^{s}_{r}$ for $-\frac12 <s<\frac1r$ and $P_{\eta_N}\,b^N$ reduces to a finite sum of terms of the form
	\begin{align}\label{eq:CDF}
	(P_{\eta_N}\,\mathds{1}_{[x_1,x_2)})(x)
	=e^{-\eta_N}\Phi\big((x_2-x)\eta_N^{-1/2}\big)-e^{-\eta_N}\Phi\big((x_1-x)\eta_N^{-1/2}\big),
	\end{align}
where $\Phi$ is the cumulative distribution of a standard normal and $x_1<x_2$ are suitable real numbers (see Appendix \ref{sc: haar} for details and Figure \ref{fig: haar wavelet} for an illustration of a Haar wavelet). This procedure introduces no additional numerical complication and suggests that Haar wavelets are a natural candidate for a numerical implementation of the scheme.

\begin{figure}
	\centering
	\begin{tikzpicture}
	\draw[->] (-0.5,0)--(7,0) node[below right] {$x$};
	\draw[->] (0.5,-1.6)--(0.5,1.6) node[left] {$h_{j,m}(x)$};
	\draw (1.5,1)  node[anchor=east]{}--(3.5,1);
	\draw (0.5,-1)[dashed] node[anchor=east]{$-1$} ;
	\draw (0.5,1)[dashed]  node[anchor=east]{$1$} ;
	\draw (3.5,-1) node[anchor=east]{} --(5.5,-1);
	\draw[dashed] (1.5,0) node[anchor=north]{$\frac{m}{2^j}$} --(1.5,1);
	\draw[dashed] (3.5,-1) node[anchor=north west]{} --(3.5,1);
	\draw[dashed] (5.5,0) node[anchor=south]{$\frac{m+1}{2^j}$} --(5.5,-1);
	\draw[color=blue, very thick] (-0.5,0) -- (1.5,0) ;
	\draw[color=blue, very thick] (1.5,1) -- (3.5,1) ;
	\draw[color=blue, very thick] (3.5,-1) -- (5.5,-1) ;
	\draw[color=blue, very thick] (5.5,0) -- (7,0) ;
	\end{tikzpicture}
	\caption{The Haar wavelet $h_{j,m}$ for $j\in\N$ and $m\in\Z$.}\label{fig: haar wavelet}
\end{figure}
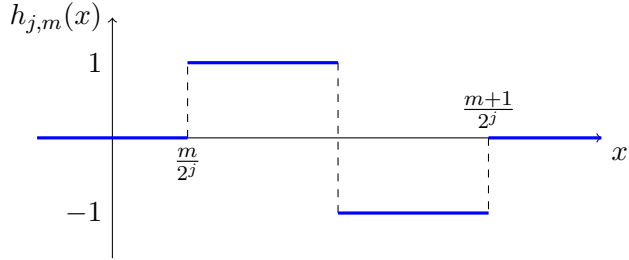

%%%%%%%%%%%%%%%%%%%%%%%%%%%%%%%%%%%%%%%%%%%%%%%%%%%%%%%%%%%%%%%%%
\section{Main theoretical results}\label{sc: main}

The main theoretical result of the paper, given in Theorem \ref{thm:global}, states a rate of convergence of the numerical scheme in an $L^1$-norm. To prove this result, we first find the rate of convergence of $X^N$ to $X$ in terms of the rate of convergence of $b^N$ to $b$ (Proposition \ref{pr:main2} and Proposition \ref{prop:Haar}). Then, for fixed $N$, we obtain the rate of convergence of the Euler-Maruyama scheme (Proposition \ref{thm:EM}). Finally, combining the two we obtain a global rate of convergence for the scheme (Theorem \ref{thm:global}).

Let us start by introducing the main assumptions on $b$ and $b^N$ which are needed for the results of this section.
\begin{assumption}\label{ass: main}
Let $ \beta_0\in\left(0,\frac{1}{4}\right)$ and $q_0\in\big(4,\frac{1}{\beta_0}\big)$ and fix $\tilde{q}_0:=(1-\beta_0)^{-1}$. For some $\kappa\in(\tfrac{1}{2},1)$ we take  $b\in  \mathcal{C}^\kappa([0,T];H^{-\beta_0}_{ \tilde q_0, q_0}) $.
\end{assumption}
Notice in particular that $b\in  \mathcal{C}^\kappa([0,T];H^{-\beta_0}_{ 2})$ by interpolation, since $1<\tilde q_0<4/3$. For future reference we also set 
\begin{equation}\label{eq:gamma0}
\gamma_0:=1-\beta_0-\frac{1}{q_0}
\end{equation}
and notice that under Assumption \ref{ass: main} we have $\gamma_0>\tfrac{1}{2}$.

\noindent {\bf Example}.
A simple example of a (time-homogeneous) drift $b$ that satisfies our Assumption \ref{ass: main} is constructed as follows: $b(t,x)=f'(x)$, where $f'$ is the distributional derivative of a $\alpha$-H\"older continuous function $f$ with compact support and $\alpha \in (\frac34,1)$.
Indeed, arguing as in \cite[Proposition 4.1]{Elena} it can be shown that any $\alpha$-H\"older continuous function $f$, with compact support and $\alpha \in(0,1)$ belongs to $H^{\alpha'}_p(\R)$ for any $1< p < +\infty$ and $0< \alpha '< \alpha$. Therefore, $f'$ is a distribution in $H^{\alpha'-1}_{p}(\R)$ for all $1< p<\infty$ and $0<\alpha'<\alpha$ and so it satisfies our Assumption 1 if $\alpha>\alpha'>\frac34$. In that case $\beta_0=\alpha'-1\in(0,\frac14)$, as needed, and $f'\in H^{-\beta_0}_{\tilde q_0,q_0}=H^{-\beta_0}_{\tilde q_0}\cap H^{-\beta_0}_{q_0}$.  

For completeness, it is worth mentioning the two key steps in the proof of \cite[Proposition 4.1]{Elena}. Let $f$ be any $\alpha$-H\"older continuous function with compact support and $\alpha\in(0,1)$. First, using an equivalent norm (see \cite[Eq.\ (10.19)]{triebel97}) it is shown that $f\in B^{\alpha'}_{p',2}(\R)$ for $0<\alpha'<\alpha$ and $p'\geq1$, where $B^{\alpha'}_{p',2}(\R)$ is a Besov space. Then, the proof is completed by the embedding $B^{\alpha'}_{p',2}(\R)\subset H^{\alpha''}_{p''}(\R) $ (see \cite[Eq.\ (11.17)]{triebel97}) for all $p''>p'$ and $\alpha'' = \alpha' -\tfrac1{p'}+\tfrac1{p''}$.
\hfill$\square$

\begin{assumption}\label{ass: b^N}
Let $(b^N)_{N\ge 1}\subset \mathcal{C}^\frac12([0,T];H^0_{\tilde q_0,q_0})$ be such that
\[
\lim_{N\to\infty}b^N=b\quad\text{in}\quad \mathcal{C}^\frac12([0,T];H^{-\beta_0}_{q_0}).
\]
\end{assumption}
 
The rate of convergence of $X^N$ to $X$ is given in the next proposition.
\begin{proposition}\label{pr:main2}
Let Assumptions \ref{ass: main} and \ref{ass: b^N} hold. 
Take any $(\beta,q)$ such that $\beta\in(\beta_0,\tfrac{1}{2})$ and $q_0\ge q>\tilde q\ge \tilde q_0$, where $\tilde q:=(1-\beta)^{-1}$. Then, for any $1/2<\gamma<\gamma_0$ there is a constant $C_{\gamma}>0$ such that 
\begin{align}\label{eq:main2}
\sup_{0\le t \le T}\E&\left[|X^N_t\!-\!X_t|\right]\le C_{\gamma} \|P_{\eta_N}\,b^N-b\|^{2\gamma-1}_{\infty,H^{-\beta}_q}  
\end{align}
as $N\to\infty$.
\end{proposition}
The proof of this result builds on a number of lemmas and we give it in Section \ref{sec:thm1} (the constant $C_{\gamma}>0$ is found explicitly). It is worth noticing that on the right-hand side of \eqref{eq:main2} we use the $H^{-\beta}_q\!$-norm with $( \beta, q)$ possibly different from $ (\beta_0, q_0)$. The reasons for this will become clear later (see in particular Proposition \ref{prop:Haar}) and in the next remark we show that the right-hand side of \eqref{eq:main2} is well-defined.   

\begin{remark}\label{rm:b^N}
Let us fix $(\beta_0, q_0)$ according to Assumption \ref{ass: main} and let us pick $(\beta, q)$ such that $\beta\in[\beta_0,\frac12)$ and $q_0\ge q> \tilde q\ge \tilde q_0$, where $\tilde q :=(1-\beta)^{-1}$. This is always possible thanks to Assumption \ref{ass: main} and, as a special case, we can pick $q=2$ as needed in Proposition \ref{prop:Haar} below. In this setting:
\begin{itemize}
\item[(i)] 
The embedding  $H^{-\beta_0}_{\tilde q_0,q_0}  \subset H^{-\beta}_{\tilde q,q}$ holds. Indeed $H^{-\beta_0}_{\tilde q_0,q_0}\subset H^{-\beta}_{\tilde q_0,q_0}$ by embedding of Sobolev spaces of negative order, see \eqref{eq:emb}, and we  also have $H^{-\beta}_{\tilde q_0,q_0}\subset H^{-\beta}_{\tilde q,q}$ by interpolation of $L^p$-spaces. Combining the above we have 
\[
b\in\mathcal{C}^\kappa([0,T]; H^{-\beta_0}_{\tilde q_0, q_0})\subset \mathcal{C}^\kappa([0,T]; H^{-\beta}_{\tilde q, q})
\] 
and, in particular, $b \in \mathcal{C}^\kappa([0,T]; H^{-\beta}_{q})$. Thanks to \eqref{eq:IZ1} also $P_{\eta_N}b $ belongs to the same space as $b$. Similarly, $b^N, P_{\eta_N}b^N \in \mathcal{C}^{\frac12}([0,T]; H^0_{\tilde q_0, q_0}) $.
\item[(ii)]  
For $(b^N)_{N\ge 1}$ as in Assumption \ref{ass: b^N} we have $P_{\eta_N}b^N \to b $ in $\mathcal C ([0,T]; H^{-\beta}_{q})$ as $N\to\infty$. Indeed, $b$, $P_{\eta_N}b$ and $P_{\eta_N}b^N$ belong to $ H^{-\beta}_{q}$ due to the item above and, thanks to \eqref{eq:IZ1}, we have 
\begin{align*}
\|P_{\eta_N}b^N - b\|_{\infty,H^{-\beta}_{q}}  
& \leq \|P_{\eta_N}(b^N-b)\|_{\infty,H^{-\beta}_{q}}  +\|P_{\eta_N}\, b-b\|_{\infty,H^{-\beta}_{q}} \\
&\leq \|b^N-b\|_{\infty,H^{-\beta}_{q}} + \|P_{\eta_N}\, b-b\|_{\infty,H^{-\beta}_{q}}  \\
&\leq \|b^N-b\|_{\infty,H^{-\beta}_{q}} + c\,\eta_N^{\frac{\beta-\beta_0}{2}} \|b\|_{\infty,H^{-\beta_0}_{q}}. 
\end{align*}
If $\beta>\beta_0$, the last term clearly goes to zero as $\eta_N\to 0$.
Moreover, since $-\beta_0-\tfrac1{q_0} \geq -\beta -\tfrac1{q}$ by assumption, using \eqref{eq:emb} we have $H^{-\beta_0}_{q_0} \subset H^{-\beta}_{q}$. Hence  
\[
\|b^N-b\|_{\infty,H^{-\beta}_{q} }  \leq c \|b^N-b\|_{\infty,H^{-\beta_0}_{q_0}}\to 0\quad\text{as} \, \, N \to \infty.
\] 
\end{itemize}
\end{remark}

In order to obtain a convergence rate in our scheme as we let $(N,m)\to\infty$ simultaneously we need to write the right-hand side of \eqref{eq:main2} explicitly in terms of $N$. For that we define a specific sequence $(b^N)_{N\geq 1}$ that satisfies Assumption \ref{ass: b^N}. 
In particular, the approximating sequence $(b^N)_{N\ge 1}$ is defined via a suitable truncation of the series expansion of $b$ in Haar wavelets. Let 
\[
\{h_{j,m},\, j\in\N\cup\{-1\},\, m \in\Z\}
\]
be a system of Haar wavelets on $\mathbb{R}$ (Definition \ref{def: Haar system R}). 
Since $b(t)\in H^{-\beta}_{q}$ for any $\beta\in[\beta_0,\tfrac{1}{2})$, $q\in[\tilde q_0,q_0]$ and all $t\in[0,T]$, by Remark \ref{rm:b^N} part (i), then we have (see \cite[Theorem 2.9]{Tri-bas}  or \eqref{eq:muh} and  Theorem \ref{thm: haar R})
\begin{align*}
&b(t) = \sum_{j=-1}^{+\infty}\sum_{m\in \Z}\mu_{j,m}(t) 2^{-j(-\beta-{\frac{1}{q}})}  h_{j,m},
\end{align*}
with $\mu_{j,m}(t)$ defined via the dual pairing of $b(t)$ and $h_{j,m}$ as
\begin{align}\label{eq:mujm}
\mu_{j,m}(t)= 2^{j(-\beta-\frac{1}{q}+1)}\langle b(t),h_{j,m}\rangle,
\end{align}
for each $t\in[0,T]$. We remark that $\mu_{j,m}(t)=\mu_{j,m}(t;\beta,q)$ only depends on $\beta$ and $q$ via the exponential term and not via the dual pairing (see Remark \ref{rm: haar coeff pq}). Later on we will use
\begin{equation}\label{eq:mus}
\mu_{j,m}(t;\beta,q)=2^{-j(\beta-\beta_0)}\mu_{j,m}(t;\beta_0,q). 
\end{equation}

From \cite[Theorem 2.9 and eq.\ (2.114), Sec.\ 2.2.3]{Tri-bas} we have
\begin{align}\label{eq:equivn-a}
\|b\|_{\infty,H^{-\beta}_q}<\infty\iff\sup_{t\in[0,T]}\|\boldsymbol{\mu}(t;\beta,q)|f^-_{q,2}\|<\infty,
\end{align}
where $\boldsymbol{\mu}(t;\beta,q) := \{\mu_{j,m}(t;\beta,q)\}_{j,m}$ and a definition of its $f^-_{q,2}$-norm with further details is provided in Appendix \ref{sc:coeffs} for completeness. In particular, if $q=2$ we have 
\begin{align}\label{eq:tribnorm2a}
\sup_{t\in[0,T]}\|\boldsymbol{\mu}(t;\beta,2)|f^-_{2,2}\|=\sup_{t\in[0,T]}\Big[ \sum_{j=-1}^\infty \sum_{m\in\Z}\big|\mu_{j,m}(t;\beta,2)\big|^2\Big]^{\frac{1}{2}}<\infty.
\end{align}
Thanks to \eqref{eq:tribnorm2a}, for $\beta\in(\beta_0,\tfrac{1}{2})$ and $N\in\N$ fixed, we can define $\tau(N)\in\N$ as the smallest integer for which
\begin{align}\label{eq:max}
\sup_{t\in[0,T]}\sum_{|m| >\tau(N)}\big|\mu_{j,m}(t;\beta,2)\big|^2\leq \frac{2^{-(N+1)(\beta-\beta_0)}}{N+1} \sup_{t\in[0,T]}\big\|\boldsymbol{\mu}(t;\beta_0,2)|f^-_{2,2}\big\|^2,
\end{align}
for all $j=0, \cdots, N$. Then we set
\begin{equation}\label{eq:haarexp}
b^N(t) := \sum_{j=-1}^{N}\sum_{m=-\tau(N)}^{\tau(N)}\mu_{j,m}(t) 2^{-j(-\beta-\frac{1}{q})} h_{j,m},
\end{equation}
where for $j=-1$ the sum in $m$ only takes the term with $m=0$. 

Clearly $b^N(t)\in  H^{0}_{\tilde q_0, q_0} \subset H^{-\beta_0}_{\tilde q_0, q_0}$ by construction. Since $b$ is $\kappa$-H\"older continuous in time with values in $H^{-\beta_0}_{\tilde q_0,q_0}$, then also the coefficients $t \mapsto \mu_{j,m}(t)$ are $\kappa$-H\"older continuous with values in $\mathbb R$. Thus we have $b^N \in \mathcal C^\kappa([0,T]; H^0_{\tilde q_0, q_0})$. 
Now we have a simple way of estimating the rate of convergence of $b^N$ to $b$ in the space $\mathcal C([0,T];H^{-\beta}_2)$, as illustrated below.
\begin{proposition}\label{prop:Haar}
Let Assumption \ref{ass: main} hold and let the sequence $(b^N)_{N\ge 1}$ be defined as in \eqref{eq:haarexp}. Then $(b^N)_{N\ge 1}$ satisfies Assumption \ref{ass: b^N} and for any $\beta\in(\beta_0,\tfrac{1}{2})$ we have
\begin{align}\label{eq:ratebN}
\big\|b^N-b\big\|_{\infty,H^{-\beta}_2}\le c\, 2^{-(N+1)(\beta-\beta_0)}\big\|b\big\|_{\infty,H^{-\beta_0}_2}. 
\end{align}
\end{proposition}
\begin{proof}
As already observed $(b^N)_{N\ge 1}\subset \mathcal C^\kappa ([0,T];H^0_{\tilde q_0,q_0})$ by construction. Recall also that $b(t), b^N(t) \in H^{-\beta}_2$ for all $\beta\in(\beta_0,\tfrac{1}{2})$ by Remark \ref{rm:b^N}, part (i).

Thanks to \eqref{eq:equivn-a} and \eqref{eq:haarexp} (see also \eqref{eq:tribnorm2} and \eqref{eq:equivn} in Appendix \ref{sc:coeffs}) it is immediate to see that $\|b^N(t)-b(t)\|_{H^{-\beta_0}_{q_0}}$ decreases to zero as $N\to\infty$. Moreover $t\mapsto \|b^N(t)-b(t)\|_{H^{-\beta_0}_{q_0}}$ is continuous and therefore by Dini's theorem $b^N\to b$ in $\mathcal C([0,T];H^{-\beta_0}_{q_0})$ as $N\to \infty$. 
{Thanks to the equivalence of the norms \eqref{eq:equivn-a} it is immediate to see that 
\begin{align*}
\sup_{t\neq s\in[0,T]}\frac{\|b^N(t)-b^N(s)\|_{H^{-\beta_0}_{q_0}}}{|t-s|^\kappa}\sim \sup_{t\neq s\in[0,T]}\frac{\|\boldsymbol{\mu}^N(t;\beta_0,q_0)-\boldsymbol{\mu}^N(s;\beta_0,q_0)|f^-_{q_0,2}\|}{|t-s|^\kappa},
\end{align*}
where we use the symbol ``$\sim$'' to indicate equivalence of the norms and $\boldsymbol{\mu}^N(\cdot ;\beta_0,q_0)=\{\mu_{j,m}(\cdot;\beta_0,q_0), j=-1,\ldots, N, m=-\tau(N),\ldots,\tau(N)\}$ contains the coefficients that appear in the expression for $b^N(\cdot)$ in \eqref{eq:haarexp}. Since the series expansion of $b^N$ contains a finite subset of the terms in the series expansion of $b$, then 
\begin{align*}
&\sup_{t\neq s\in[0,T]}\frac{\|\boldsymbol{\mu}^N(t;\beta_0,q_0)-\boldsymbol{\mu}^N(s;\beta_0,q_0)|f^-_{q_0,2}\|}{|t-s|^\kappa}\\
&\le \sup_{t\neq s\in[0,T]}\frac{\|\boldsymbol{\mu}(t;\beta_0,q_0)-\boldsymbol{\mu}(s;\beta_0,q_0)|f^-_{q_0,2}\|}{|t-s|^\kappa}\le c  [b]_{\kappa,H^{-\beta_0}_{q_0}},
\end{align*}
for some constant $c>0$ that arises from the equivalence of norms.}
%Since $(b^N)_{N\ge 1}\subset \mathcal C^\kappa([0,T];H^{-\beta_0}_{q_0})$ and
%\[
%\lim_{N\to\infty}\left( \sup_{\substack{t,s\in[0,T]\\|t-s|\le \delta}}\|b^N(t)-b^N(s)\|_{H^{-\beta_0}_{q_0}}\right) = \sup_{\substack{t,s\in[0,T]\\|t-s|\le \delta}}\|b(t)-b(s)\|_{H^{-\beta_0}_{q_0}}\le [b]_{\kappa,H^{-\beta_0}_{q_0}}\delta^\kappa,
%\]
Then $(b^N)_{N\ge 1}$ is also a bounded subset of $\mathcal C^\kappa([0,T];H^{-\beta_0}_{q_0})$. Recalling that $\kappa\in(\tfrac{1}{2},1)$ we can conclude that $b^N\to b$ in $\mathcal C^\frac{1}{2}([0,T];H^{-\beta_0}_{q_0})$ as $N\to \infty$ since the embedding $\mathcal C^{\frac{1}{2}}\subset \mathcal C^\kappa$ is compact. Hence, $(b^N)_{N\ge 1}$ satisfies Assumption \ref{ass: b^N}.

It remains to prove \eqref{eq:ratebN}. For $q=2$ and a suitable constant $c>0$, using \eqref{eq:mus}, \eqref{eq:tribnorm2a} and \eqref{eq:max} we have
\begin{align*}
& \big\|b^N(t)-b(t)\big\|^2_{H^{-\beta}_2}\\
&\le c\Big(\sum_{j>N}\sum_{m\in\Z}\big|\mu_{j,m}(t;\beta,2)\big|^2 + \sum_{j=0}^N\sum_{|m|>\tau(N)}\big|\mu_{j,m}(t;\beta,2)\big|^2\Big) \\
&\le c\, \Big(2^{-2(N+1)(\beta-\beta_0)} \sum_{j>N}\sum_{m\in\Z}\big|\mu_{j,m}(t;\beta_0,2)\big|^2+2^{-2(N+1)(\beta-\beta_0)} \big\|\boldsymbol{\mu}(t;\beta_0,2)|f^-_{2,2}\big\|^2\Big)\\
&\le 2\, c\, 2^{-2(N+1)(\beta-\beta_0)}\big\|\boldsymbol{\mu}(t;\beta_0,2)|f^-_{2,2}\big\|^2.
\end{align*}  
Taking supremum over $t\in[0,T]$ and recalling \eqref{eq:equivn-a} we conclude by incorporating the factor 2 into the constant $c$.
\end{proof}
Since the drift in the SDE for $X^N$ is Lipschitz in space and $1/2$-H\"older continuous in time we expect a standard strong convergence rate of $1/2$ for the Euler-Maruyama scheme. This is confirmed in the next proposition where, however, we are particularly interested in the dependence of the multiplicative constants on $N$. By controlling those constants, later on, we will establish an overall rate of convergence for the scheme as we let $(N,m)\to\infty$ at the same time.
\begin{proposition}\label{thm:EM}
Let Assumption \ref{ass: main} hold and let $b^N\in \mathcal{C}^{\frac{1}{2}}([0,T];H^0_{\tilde q_0, q_0})$ for some fixed $N$. Then, as $m\to\infty$, we have
\begin{align}\label{eq:EM}
\sup_{0\le t \le T}\E\left[|X^{N,m}_t\!-\!X^N_t|\right]\le C_2(N) m^{-1}+C_3(N) m^{-\frac{1}{2}} 
\end{align}
with
\begin{align}\label{eq:C23}
\begin{split}
&C_2(N):=c\,\|P_{\eta_N} b^N\|_{\infty,L^\infty}\Big(1 +\|\nabla (P_{\eta_N} b^N)\|_{\infty,L^\infty}\Big),\\
&C_3(N):=c'\,\Big(\|\nabla (P_{\eta_N} b^N)\|_{\infty,L^\infty}+[P_{\eta_N} b^N]_{\frac12,L^\infty}\Big)
\end{split}
\end{align}
and $c,c'>0$ constants independent of $(N,m)$.
\end{proposition}
\noindent The proof of the proposition is given in Section \ref{sec:thm2}.

Combining the results above we obtain the full convergence result, that summarises the theoretical  findings in the paper.
\begin{theorem}\label{thm:global}
Let Assumption \ref{ass: main} hold, let $(b^N)_{N\ge 1}$ be defined as in \eqref{eq:haarexp} (so that Assumption \ref{ass: b^N} holds too) and let $\theta_*:=\tfrac{1}{2}\big[\tfrac{3}{4}-\beta_0(\gamma_0-\tfrac{1}{2})\big]^{-1}$ with $\gamma_0$ as in \eqref{eq:gamma0}. Then, as $m\to \infty$, taking $\eta_N=m^{-\theta_*}$ and $N= \lfloor 2\theta^* \log_2 m \rfloor$ we have
\begin{align}\label{eq:global}
\sup_{0\le t \le T}\E\left[|X^{N,m}_t\!-\!X_t|\right]\le c_\eps\bigg(m^{-\theta_*(\frac{1}{2}-\beta_0)(\gamma_0 - \frac{1}{2})+\eps}\bigg),
\end{align}
where $\varepsilon>0$ can be arbitrarily small and $c_\eps>0$ is a constant depending on $\eps$.
\end{theorem}
\begin{figure}
\centering
\begin{subfigure}[b]{0.475\textwidth}
\centering
	\includegraphics[width =\linewidth]{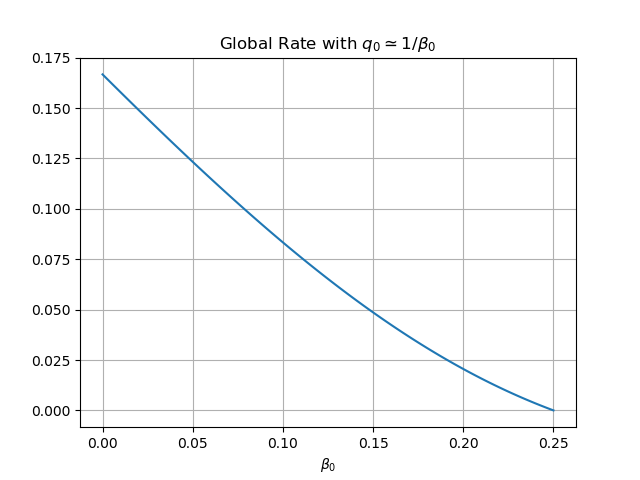}
	\caption{Convergence rate for  $q_0 \approx 1/\beta_0$
	 }
	\label{fig: rate of beta}
\end{subfigure}
\begin{subfigure}[b]{0.475\textwidth}
\centering
	\includegraphics[width =  \linewidth]{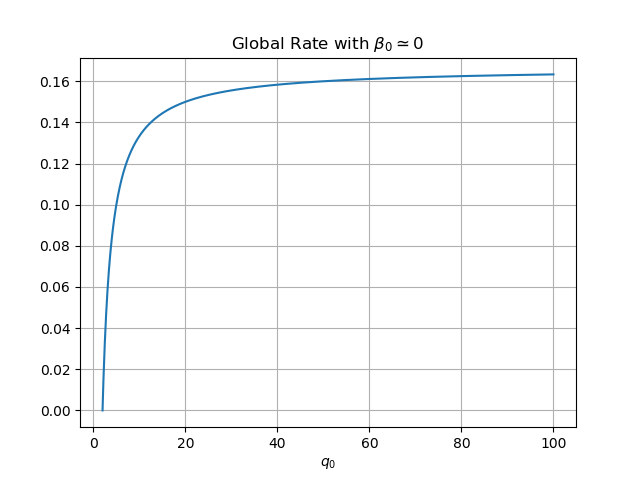}
	\caption{Convergence rate for   $\beta_0\approx0$}
	\label{fig: rate of q0}
\end{subfigure}
\caption{Convergence rate of the scheme as function of the parameter $\beta_0\in(0, 1/4)$ (left panel) and $q_0>4$ (right panel) when the other parameter is fixed.} 
\label{fig:rate}
\end{figure}
Before proving the theorem, we offer some basic insight into the meaning of the rate in \eqref{eq:global} (in Figure \ref{fig:rate} we also plot two examples).
\begin{remark}\label{rem:Lp}
\, 
\begin{itemize}
\item If $q_0\approx 1/\beta_0$, which is the largest possible $q_0$, then the rate decreases as $\beta_0$ increases, and the best rate is obtained when $\beta_0$ is close to zero. The rate  is illustrated in Figure~\ref{fig: rate of beta} as a function of $\beta_0\in (0,\frac14)$. 
\item For $\beta_0\approx 0$ we have $\tilde q_0\approx1$ and we can pick any $4<q_0<\infty$. Then we obtain a convergence rate for SDEs with drift in $\mathcal{C}^{\kappa}([0,T];L^{1} \cap L^{q_0})$ for $\kappa >\tfrac{1}{2}$.  The rate  is illustrated in Figure~\ref{fig: rate of q0} as a function of $q_0 >4$. While existence of strong solutions for SDEs with drift in $L^p ([0,T];L^q)$-spaces was obtained by Krylov and R\"ockner in \cite{Kry-Ro}, we are not aware of convergence rates results for numerical schemes if $b\in\mathcal C^{\kappa}([0,T]; L^{q_0})$ and $1<q_0<\infty$.
\item In the extreme case when $\beta_0\approx 0$ and $q_0\to\infty$ (i.e.\ $b\in\mathcal{C}^\kappa([0,T]; L^1\cap L^\infty)$ with $\kappa>\frac{1}{2}$) we obtain a convergence rate of $m^{-1/6}$. However, if we assumed $b\in\mathcal C^\kappa([0,T]; L^\infty)$ from the start, we would have been able to obtain a better rate from the Euler-Maruyama scheme using ideas from Dareiotis and Genrencs\'er \cite{dareiotis2020} (notice though that the constants in \cite[Lemma 2.2]{dareiotis2020} depend exponentially on $\|b\|_{L^\infty}$). The approach we take in Section \ref{sec:thm2} allows us to avoid that the constants in Proposition \ref{thm:EM} depend exponentially on the $L^\infty$-norm of $b^N$, which is essential when $b$ is a distribution. Clearly, if $b\in \mathcal{C}^\kappa([0,T]; L^\infty)$ that caution is no longer needed.
\end{itemize}
\end{remark}

\begin{proof}[Proof of Theorem \ref{thm:global}]
Fix $\eps>0$ and let $c_\eps>0$ be a constant that may vary from line to line, possibly depending on $\|b\|_{\infty,H^{-\beta_0}_{\tilde q_0,q_0}}$ and $\eps>0$ but independent of $N$ and $m$. In the rest of this proof we will use Proposition \ref{pr:main2} with $\gamma\!=\!\gamma_0\!-\!\tfrac{1}{2}\eps$ so that the constant $C_{\gamma}$ is absorbed in $c_\eps>0$.

Using triangular inequality, \eqref{eq:main2} from Proposition \ref{pr:main2} (with $q=2$) and \eqref{eq:EM} from Proposition \ref{thm:EM} we obtain
\begin{align}\label{eq:globr0}
&\sup_{0\le t \le T}\E\left[|X^{N,m}_t\!-\!X_t|\right]\\
&\le c_\eps\left(\|P_{\eta_N}(b^N-b)\|^{2\gamma_0-1-\eps}_{\infty,H^{-\beta}_{2}}+\|P_{\eta_N}\,b-b\|^{2\gamma_0-1-\eps}_{\infty,H^{-\beta}_{2}}+C_2(N)m^{-1}+C_3(N)m^{-\frac{1}{2}}\right).\notag
\end{align}

The $L^\infty$-norms appearing in the constants $C_2(N)$ and $C_3(N)$ in \eqref{eq:C23} can be estimated further by using \eqref{eq:IZ2}. 
Recall that $(b^N)_{N\ge 1}$ fulfils Assumptions \ref{ass: b^N} thanks to Proposition \ref{prop:Haar}. The most favourable estimates in \eqref{eq:IZ2} are obtained for $r=q_0$, $s=\beta_0$ and $\nu=\tfrac{1}{q_0}+\eps$. Then, we have
\begin{equation}\label{eq:IZ3}
\begin{split}
&\|P_{\eta_N} b^N\|_{\infty,L^\infty}\le c\,\eta_N^{-\frac{1}{2}(\beta_0+\frac{1}{q_0}+\eps)}\|b^N\|_{\infty,H^{-\beta_0}_{q_0}}\le c\,\eta_N^{-\frac{1}{2}(1-\gamma_0+\eps)}\|b\|_{\infty,H^{-\beta_0}_{q_0}}\\
&\|\nabla(P_{\eta_N} b^N)\|_{\infty,L^\infty}\le c\,\eta_N^{-\frac{1}{2}(1+\beta_0+\frac{1}{q_0}+\eps)}\|b^N\|_{\infty,H^{-\beta_0}_{q_0}}\le c\,\eta_N^{-\frac{1}{2}(2-\gamma_0+\eps)}\|b\|_{\infty,H^{-\beta_0}_{q_0}},\\
&\big[P_{\eta_N}b^N\big]_{\frac12,L^\infty}\le c\,\eta_N^{-\frac{1}{2}(\beta_0+\frac{1}{q_0}+\eps)}\big[b^N\big]_{\frac12,H^{-\beta_0}_{q_0}}\le c\,\eta_N^{-\frac{1}{2}(1-\gamma_0+\eps)}\big[b\big]_{\frac12,H^{-\beta_0}_{q_0}}, 
\end{split}
\end{equation}
where the final inequality in each of the above expressions follows from the convergence $\lim_{N\to\infty}b^N = b$ in $\mathcal{C}^{\frac{1}{2}}([0,T];H^{-\beta_0}_{q_0})$ and we used $\beta_0+1/q_0=1-\gamma_0$ and $1+\beta_0+1/q_0=2-\gamma_0$.

Using also \eqref{eq:IZ1} to bound the first two terms on the right-hand side of \eqref{eq:globr0} (where $c_\eps$ may change from line to line)  we have
\begin{align}\label{eq:E} 
\nonumber
&\sup_{0\le t \le T}\E\left[|X^{N,m}_t\!-\!X_t|\right]\\ 
&\le c_\eps\bigg(\|b^N-b\|^{2\gamma_0-1-\eps}_{\infty,H^{-\beta}_{2}}+\eta_N^{(\beta-\beta_0)(\gamma_0-\frac{1}{2})-\eps}\\\nonumber
&\:\:\qquad+\eta_N^{-\frac{1}{2}(3-2\gamma_0+\eps)}m^{-1}+\eta_N^{-\frac{1}{2}(2-\gamma_0+\eps)}m^{-\frac{1}{2}}\bigg),
\end{align}
where for the final two terms we selected the leading order in $\eta_N$ by using that $\eta_N\in(0,1)$, with no loss of generality, and $1+\beta_0+\tfrac{1}{q_0}\ge 2(\beta_0+\tfrac{1}{q_0})$, since $\beta_0\in(0,\tfrac{1}{4})$.

Thanks to Proposition \ref{prop:Haar}, and abusing slightly the notation by letting $\eps>0$ vary from the first to the second inequality, we have 
\begin{align}\label{eq:E2}
\|b^N-b\|^{2\gamma_0-1-\eps}_{\infty,H^{-\beta}_{2}}&\le c\, 2^{-(N+1)(\beta-\beta_0) (2\gamma_0-1-\eps)} \leq c\, 2^{-N\frac12 (\beta-\beta_0)(\gamma_0-\frac12)+\eps}.
\end{align}
The aim is to let $N$ and $m$ diverge to infinity and $\eta_N\to 0$ at the same time. In order to do so we  choose suitable $N$   and $\eta_N$ depending on $m$. Take $\eta_N=m^{-\theta}$ for some $\theta>0$ to be determined. The last three terms in \eqref{eq:E} read
\begin{align}\label{eq:mm}
 m^{-\theta(\beta-\beta_0)(\gamma_0-\frac{1}{2})+\eps}+m^{-1+\frac{1}{2}\theta(3-2\gamma_0+\eps)}+m^{-\frac{1}{2}+\frac{1}{2}\theta(2-\gamma_0+\eps)}
\end{align}
and, as $m\to\infty$, the leading terms are the first and last one. By comparing \eqref{eq:E2} and \eqref{eq:E} we notice that there is no loss of generality in choosing $N = \lfloor2 \theta \log_2 m \rfloor$. Finally, plugging \eqref{eq:E2} and \eqref{eq:mm} back into \eqref{eq:E}, ignoring terms of lower order in $m$, we obtain
\begin{align*}
\sup_{0\le t \le T}\E\left[|X^{N,m}_t\!-\!X_t|\right]&\le c_\eps\bigg(m^{-\theta(\beta-\beta_0)(\gamma_0-\frac{1}{2})+\eps}+m^{-\frac{1}{2}+\frac{1}{2}\theta(2-\gamma_0+\eps)}\bigg)\\
&\le c_\eps\bigg(m^{-\theta(\frac{1}{2}-\beta_0)(\gamma_0-\frac{1}{2})+\eps}+m^{-\frac{1}{2}+\frac{1}{2}\theta(2-\gamma_0)+\eps}\bigg),
\end{align*}
where for the second inequality we have chosen the best possible $\beta$, which is just below $\frac12$, and with a slight abuse of notation we have allowed $\eps$ to vary from line to line. It remains to select $\theta>0$ that gives the fastest convergence rate. Notice that the first term on the right-hand side of the expression above is decreasing in $\theta$ whereas the second one is increasing. Then the optimum is attained when the exponents are equal and we get
\[
\theta_*=\tfrac{1}{2}\Big[\tfrac{3}{4}-\beta_0(\gamma_0-\tfrac{1}{2})\Big]^{-1},
\]
as claimed.
\end{proof}

\begin{remark}
For practical use of our numerical scheme one must compute the coefficients $\mu_{j,m}$ from \eqref{eq:mujm}. It is shown in Theorem \ref{thm: faber I} and Remark \ref{rmk: haar faber link} that such computation is very easy when $b(t,\cdot)$ is supported on a bounded interval on $\R$. 
\end{remark}

\begin{remark}
When implementing the scheme, the coefficients in the formula for $b^N$ are computed offline and stored in the memory at the beginning of the algorithm. The complexity of the algorithm is determined by the number of operations involving such coefficients and the Haar functions (multiplications and summation) and by the number of time steps in the Euler scheme. In particular, we count the number of terms in the sum and we multiply that by the number of time-steps in the Euler scheme. From \eqref{eq: scheme} and \eqref{eq:haarexp}, the overall complexity of the algorithm is $O(m N \tau(N))$. Hence, by taking $N= 2\theta^* \log_2 m$ as in the statement of Theorem \ref{thm:global} we obtain a complexity of $O(m \log_2 m\ \tau(\lfloor 2\theta^* \log_2 m\rfloor))$. Unfortunately $\tau(N)$ can be difficult to compute in general but, in the special case of $b$ supported on a bounded interval $I$, we have 
\begin{equation*}
b(t) = \mu_0(t) h_0 + \sum_{j=0}^{+\infty}\sum_{m=0}^{2^j-1}\mu_{j,m}(t)2^{-j\left({-\beta}-\frac{1}{q}\right)}h_{j,m}
\end{equation*} 
(see \eqref{eq: haar represent Ia} in Appendix \ref{sc: haar}, where we take $I=(0,1)$ for simplicity and with no loss of generality). 
Then we can define $b^N$ as
\begin{equation}
b^N(t) = \mu_0(t) h_0 + \sum_{j=0}^{N}\sum_{m=0}^{2^j-1}\mu_{j,m}(t)2^{-j\left({-\beta}-\frac{1}{q}\right)}h_{j,m},
\end{equation}
and, in the proof of Proposition \ref{prop:Haar}, we have 
\begin{align*}
\big\|b^N(t)-b(t)\big\|^2_{H^{-\beta}_2}\le c\sum_{j>N}\sum_{m=0}^{2^j-1}\big|\mu_{j,m}(t;\beta,2)\big|^2\le c\, 2^{-2(N+1)(\beta-\beta_0)}\big\|\boldsymbol{\mu}(t;\beta_0,2)|f^-_{2,2}\big\|^2.
\end{align*}
In that case the complexity is $O(m 2^N)$ hence by taking $N= \lfloor 2\theta^* \log_2 m \rfloor$ we obtain a complexity of $O(m^2)$. Notice that the computation of the semigroup in \eqref{eq:CDF} does not modify the complexity.
\end{remark}

%%%%%%%%%%%%%%%%%%%%%%%%%%%%%%%%%%%%%%%%%%%%%%%%%%%%%%%%%%%%%%%%%

\section{Background material on virtual solutions}\label{sec:background}

As anticipated, the proofs of Proposition \ref{pr:main2} and Proposition \ref{thm:EM} rely upon a few technical lemmas. To set out clearly our arguments and keep the exposition self-contained it is convenient to review and complement some results from \cite{Fla-Iss-Rus-2017}. 

We will work in the framework of \cite{Fla-Iss-Rus-2017} but we restrict our attention to $[0,T]\times\R$ rather than working with $[0,T]\times\R^d$ as in the original paper. Throughout this section we make the following standing assumption. 
\begin{assumption}\label{ass:mainFIR}
Let $ \beta\in\left(0,\frac{1}{2}\right)$, fix $\tilde{q}:=\frac{1}{1-\beta}$ and let $q\in\big(\tilde{q},\frac{1}{\beta}\big)$. We take $b\in  \mathcal{C}([0,T];H^{-\beta}_{\tilde q, q}) $.
\end{assumption} 
Notice that Assumption \ref{ass:mainFIR} is implied by Assumption \ref{ass: main} (with $(\beta_0,q_0)$ instead of $(\beta,q)$ and $q>4$). It was shown in \cite[Theorem 28]{Fla-Iss-Rus-2017} that under Assumption \ref{ass:mainFIR} for every $x\in \mathbb R$ there exists a unique in law \emph{virtual solution} of \eqref{eq: original sde}. A virtual solution of \eqref{eq: original sde}  is given in terms of a stochastic basis
$\left(  \Omega,\mathcal{F},\mathbb{F},\mathbb P,W\right)$ and an $\mathbb{F}$-adapted, continuous
stochastic process $X:=(X_{t})_{t\in[0,T]}$ (shortened as $\left(X,\mathbb{F}\right)$)
such that the integral equation
\begin{align}
\label{eq: virtual solution}
X_{t}=&\ x+u(0,x)-u(t,X_{t})+(\lambda+1)\int_{0}^{t}u(s,X_{s})\di
s\\ \nonumber
&+\int_{0}^{t} {( \nabla u(s,X_{s}) + 1)}  \di W_{s}
\end{align}
holds for all $t\in[0,T]$, with probability one. Here $u$ is the mild solution of the following parabolic Kolmogorov-type PDE 
\begin{equation}\label{eq: pde}
\begin{cases}
\partial_t u + \frac{1}{2}\Delta u + b\nabla u - (\lambda+1)u = -b\ &\mathrm{on}\ [0,T]\times\R\\
u(T) = 0\ &\mathrm{on}\ \R
\end{cases}
\end{equation}
with $\lambda>0$. The mild solution $u$   is unique in $\mathcal{C}([0,T];H_p^{1+\delta})$, for any $(\delta,p)\in \mathcal{K}(\beta,q)$, where the set $\mathcal{K}(\beta,q)$ is defined as 
\begin{align}\label{eq:K}
\mathcal{K}(\beta,q):=\{(\delta,p)\ |\ \beta<\delta<1-\beta,\ \tfrac{1}{\delta}<p<q\}.
\end{align}   
The set $\mathcal{K}(\beta,q)$ is drawn in  Figure \ref{fig: the set K} for the reader's convenience and it is not empty thanks to Assumption \ref{ass:mainFIR}. Notice that the stochastic integral that appears in \eqref{eq: virtual solution} is well-defined thanks to fractional Morrey's inequality \eqref{eq:morrey}.
\begin{remark}[Uniqueness]\label{rem:uni}
We remark that, thanks to the shape of $\mathcal K(\beta,q)$ and to the embedding \eqref{eq:emb},  given two couples $(\delta_1,p_1), (\delta_2, p_2) \in \mathcal{K}(\beta,q)$, it is always possible to find $ (\delta, p) \in \mathcal{K}(\beta,q)$ such that $H^{1+\delta_1}_{p_1}\subset H^{1+\delta}_{p}$ and $H^{1+\delta_2}_{p_2}\subset H^{1+\delta}_{p}$, see  Figure \ref{fig: the set K}. Since the solution $u$ to \eqref{eq: pde} is unique in the space $ H^{1+\delta}_{p}$, it follows that it must coincide with the solutions found in the spaces $H^{1+\delta_1}_{p_1}$ and $H^{1+\delta_2}_{p_2}$. Hence, the solution of \eqref{eq: pde} is unique in the whole triangle $\mathcal K(\beta,q)$.
\end{remark}
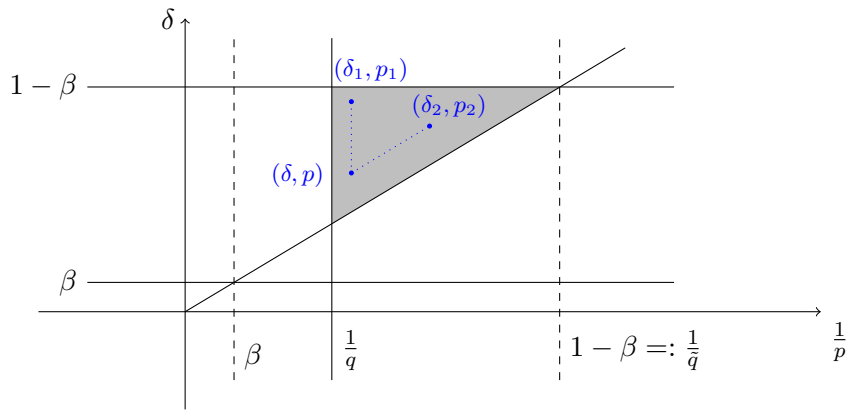
\begin{figure}
	\centering
	\begin{tikzpicture}[scale = 1.3]
	% axes	
	\draw[->] (-1,0)--(7,0) node[below right] {$\frac1p$};
	\draw [->] (0.5,-1)--(0.5,3) node[left] {$\delta$};
	
	% triagle	
	\fill[mygray] (2, .9)--(2,2.3)--(4.33,2.3)--(2, .9);
	\draw (-.5,.3)  node[anchor=east]{$\beta$}--(5.5,.3);
	\draw (-.5,2.3) node[anchor=east]{$1-\beta$} --(5.5,2.3);
	\draw (2,-.7) node[anchor=south west]{$ \frac1q$} --(2,2.8);
	\draw[dashed] (1,-.7) node[anchor=south west]{${ \beta}$} --(1,2.8);
	\draw[dashed] (4.33,-.7) node[anchor=south west]{$1-\beta=:\frac{1}{\tilde q}$} --(4.33,2.8);
	\draw (.5,0)--(5,2.7);  % I am drawing y = 3/5 x - 3/10
	
	%blue embedding
	\draw [fill, blue](2.2,1.42)  circle [radius=0.02];
	\node[blue] at (1.65,1.4) {\footnotesize{$(\delta, p)$}};
	\draw [fill, blue](2.2,2.15)  circle [radius=0.02];
	\node[blue, anchor = north] at (2.4,2.7) {\footnotesize{$(\delta_1, p_1)$}};
	\draw [fill, blue](3,1.9)  circle [radius=0.02];
	\node[blue] at (3.2,2.1) {\footnotesize{$(\delta_2, p_2)$}};
	\draw[blue, dotted]  (2.2,1.42) --(3,1.9);
	\draw[blue, dotted]  (2.2,1.42) --(2.2,2.15);	
	\end{tikzpicture}

	\caption{An illustration of the set $\mathcal{K}(\beta, q)$ (figure modified from \cite{Fla-Iss-Rus-2017}). Given a couple $(\beta,q)$ according to Assumption \ref{ass:mainFIR}, there exists a unique mild solution $u\in\mathcal C([0,T]; H^{1+\delta}_p)$ for the PDE \eqref{eq: pde}, for all $(\delta, p)$ in the grey triangle.
For any two points $(\delta_1,p_1), (\delta_2, p_2) \in \mathcal{K}(\beta,q)$, it is always possible to find $ (\delta, p) \in \mathcal{K}(\beta,q)$ such that $H^{1+\delta_1}_{p_1}\subset H^{1+\delta}_{p}$ and $H^{1+\delta_2}_{p_2}\subset H^{1+\delta}_{p}$, see dotted lines for embeddings. 	
	}\label{fig: the set K}
\end{figure}

It is worth  noticing that the concept of virtual solution follows a Zvonkin-type transformation based on heuristic application of It\^o's formula to $u(t, X_t)$. This allows to replace the drift term $b(t, X_t) \di t$ in \eqref{eq: original sde} with the terms in \eqref{eq: virtual solution} depending on $u$ and $\nabla u$. The reader might have noticed that the PDE \eqref{eq: pde} and the virtual SDE \eqref{eq: virtual solution} depend on an extra parameter $\lambda$, while the original SDE \eqref{eq: original sde} does not. This is due to a technical step in the proof, that leads to good properties of $u$. However, it is possible to show that the virtual solution is independent of $\lambda$, as shown in \cite[Section 3.3]{Fla-Iss-Rus-2017}.

For the numerical scheme illustrated  in Section \ref{sec: scheme} we also need to consider the approximating PDE
\begin{equation}\label{eq:approxpde}
\begin{cases}
\partial_t u^N + \frac{1}{2}\Delta u^N + a^N\nabla u^N - (\lambda+1)u^N = -a^N\ &\mathrm{on}\ [0,T]\times\R\\
u^N(T) = 0\ &\mathrm{on}\ \R,
\end{cases}
\end{equation}
where $a^N:=P_{\eta_N} b^N$, for each $N\ge 1$. 

We will now review the arguments that guarantee existence, uniqueness and regularity of the solutions to \eqref{eq: pde} and \eqref{eq:approxpde}. 
Under Assumption \ref{ass:mainFIR} and for $(\delta,p) \in \mathcal K(\beta,q)$, \cite[Theorem 14]{Fla-Iss-Rus-2017} guarantees that for each $\lambda>0$ there exists a unique solution $u_\lambda\in\mathcal{C}([0,T];H^{1+\delta}_p)$ to \eqref{eq: pde}.  
Since the time derivative and the second spatial derivative of $u_\lambda$ are not well defined, $u_\lambda$ is a so-called {\em mild} solution  (for details see, e.g., \cite{Elena}), and it is obtained as a fixed point in the space $\mathcal{C}([0,T];H^{1+\delta}_p)$  equipped with the norm $\norme{\cdot }^{(\rho)}_{\infty,H^{1+\delta}_p}$, with $\rho\!>\!\lambda$ sufficiently large. 
Using fractional Morrey's inequality \eqref{eq:morrey} it is possible to embed the fractional Sobolev space $H_p^{1+\delta}$ in smoother spaces. In particular we have 
\begin{align}\label{holder}
u_\lambda\in\mathcal{C}([0,T];\mathcal{C}^{1,\gamma}),\quad\text{with $\gamma=\delta-1/p$.}
\end{align} 
Analogously, \eqref{eq:approxpde} admits a unique solution $u^N_\lambda\in\mathcal{C}([0,T];H^{1+\delta}_p)$ (regularity of $u^N_\lambda$ could of course be upgraded to $\mathcal{C}^{1,2}([0,T)\times\mathbb{R})\cap \mathcal{C}([0,T]\times\mathbb{R})$ by virtue of higher regularity of $a^N$ but this will not be needed for our purposes).

Next, \cite[Lemma 20]{Fla-Iss-Rus-2017} gives useful bounds for the gradient of $u_\lambda$ and $u^N_\lambda$. We give a statement which is adapted to our notation\footnote{We note that there is a typo in the statement of \cite[Lemma 20]{Fla-Iss-Rus-2017}. Indeed it can be easily checked from the proof that the condition $\rho\!<\!\lambda$ is not needed therein.}.
\begin{lemma}\label{lem: grad}
	Let $(\delta,p)\in {\mathcal{K}}(\beta,q)$. There exists $\lambda_0\!>\!0$ such that, given any $\lambda\!>\!\lambda_0$, letting $u\!=\!u_\lambda$ and $u^N\!=\!u^N_\lambda$ be the mild solutions in $\mathcal{C}([0,T];H_p^{1+\delta})$ to the corresponding problems \eqref{eq: pde} and \eqref{eq:approxpde}, respectively, we have
	\begin{align}
	&{\sup_{(t,x)\in[0,T]\times\mathbb{R}}} |\nabla u(t,x)| \leq\tfrac{1}{2}\qquad\text{and}\qquad {\sup_{(t,x)\in[0,T]\times\mathbb{R}}} |\nabla u^N(t,x)|\leq\tfrac{1}{2}.
	\end{align}
\end{lemma}
Furthermore,  \cite[Lemma 21]{Fla-Iss-Rus-2017} also guarantees that 
\begin{align}\label{eq:gradcont}
u,\,u^N,\,\nabla u,\,\nabla u^N \in \C([0,T]\times\R).
\end{align}
The next result is a refined statement of  \cite[Lemma 23]{Fla-Iss-Rus-2017}. In particular our equation \eqref{eq:bconv} is contained in the final part of the original proof in \cite{Fla-Iss-Rus-2017}.
\begin{lemma}\label{lem:conv}
	Let $(\delta,p)\in \mathcal{K}(\beta,q)$ and let $u\!=\!u_\lambda$ and $u^N\!=\!u^N_\lambda$ be the mild solutions in $\mathcal{C}([0,T];H_p^{1+\delta})$ to \eqref{eq: pde} and \eqref{eq:approxpde}, respectively. Then, if $a^N\to b$ in $\mathcal{C}([0,T];H^{-\beta}_{q})$,  there is a constant $c_0>0$ such that 
	\begin{align}\label{eq:bconv}
	\|u-u^N\|^{(\rho)}_{\infty,H^{1+\delta}_p}\le c_0\frac{\left(\|u\|^{(\rho)}_{\infty,H^{1+\delta}_p}+1\right)\rho^{\frac{\delta+\beta-1}{2}}}{1-c_0\left(\|b\|_{H^{-\beta}_{ q}}\, \rho^{\frac{\delta+\beta-1}{2}}+\lambda\rho^{-1}\right)}\|b-{a^N}\|_{{\infty},H^{-\beta}_{ q}},
	\end{align}
	for any $\rho>\lambda$ that is sufficiently large to guarantee that the denominator above is positive.
\end{lemma}

For future reference we define
\begin{align}\label{eq:crho}
c(\rho):=\frac{c_0\left(\|u\|^{(\rho)}_{\infty,H^{1+\delta}_p}+1\right)\rho^{\frac{\delta+\beta-1}{2}}}{1-c_0\left(\|b\|_{{\infty},H^{-\beta}_{ q}} \,\rho^{\frac{\delta+\beta-1}{2}}+\lambda\rho^{-1}\right)},
\end{align}
for $\rho>0$ large enough so that the denominator is positive. 

\begin{remark}\label{rem:lN}
	Notice that in Lemma \ref{lem: grad} we can choose $N\ge N_0$, sufficiently large, so that $\lambda_0$ depends only on $\delta,\beta$ and $\|b\|_{\infty,H_{q}^{-\beta}}$, because $a^N\to b$ in $\mathcal{C}([0,T];H_{ q}^{-\beta})$.
	Then, in Lemma \ref{lem:conv} we can choose $\rho>\rho_0$ so that the denominator in \eqref{eq:bconv} is positive and $\rho_0>\lambda_0$ (as needed for the fixed point in \cite[Theorem 14]{Fla-Iss-Rus-2017}).
\end{remark}

From now on we will simplify our notation and set $u=u_\lambda$, for some $\lambda$ sufficiently large so that Lemma \ref{lem: grad} holds. In order to solve equation \eqref{eq: virtual solution} and find a virtual solution of \eqref{eq: original sde}, one has to transform the SDE \eqref{eq: virtual solution} into a more standard one. This is achieved by setting $Y_t := \varphi(t, X_t)$, where
\begin{align}\label{eq:phi}
\varphi(t,x) := x + u(t,x).
\end{align} 
Notice that $\varphi \in\mathcal C([0,T]; \mathcal C^1)$ thanks to \eqref{eq:gradcont}. Moreover by Lemma \ref{lem: grad} $x\mapsto \varphi(t,x) $ is invertible for each fixed $t\in [0,T]$, with its inverse denoted by
\begin{align}\label{eq:psi}
\psi(t,\cdot) := \varphi^{-1}(t, \cdot).
\end{align} 
By Lemma \ref{lem: grad}, $\psi(t, \cdot)$ is 2-Lipschitz, uniformly in $t$.
Then, solving \eqref{eq: virtual solution} is equivalent to solving the standard SDE for $Y$ below
\begin{align}\label{eq: Y}
Y_t =\ y_0 \!+\! (\lambda+1)\!\int_0^t\!\! u(s,\psi\left(s,Y_s\right)) \di s\!+\!\int_0^t\!\! (\nabla u(s,\psi\left(s,Y_s\right))\!+\!1) \di W_s,%\nonumber  
\end{align}
where $y_0 = \varphi(0, x)$. Existence of a weak  solution for \eqref{eq: Y} is guaranteed by \cite[Theorem 10.2.2]{stroock2007}  since its coefficients $\tilde b(t,y):=(\lambda+1)u(t,\psi(t,y))$ and $\tilde \sigma(t,y):=\nabla u(t,\psi(t,y))+1$ are bounded continuous with $\tilde \sigma$ uniformly non-degenerate (see \cite[Proposition 27]{Fla-Iss-Rus-2017} for details). 

Likewise, letting $\varphi^N(t,x) := x + u^N(t,x)$, $y^N_0:=\varphi^N(0,x)$ and $\psi^N(t,\cdot):=\left(\varphi^N\right)^{-1}(t,\cdot)$, the analogue of \eqref{eq: Y} for the approximated SDE \eqref{eq: approximated sde} is given by an SDE for $Y^N: = \psi^N(t,X^N_t)$. That is 
\begin{align}\label{eq: approximated Y} 
Y_t^N =& y_0^N\! +\! (\lambda\!+\!1)\!\int_0^t\!\! u^N\left(s,\psi^N\left(s,Y_s^N\right)\right) \di s \\
&+\int_0^t \left(\nabla u^N\left(s,\psi^N\left(s,Y_s^N\right)\right)+1\right)\di W_s. \nonumber
\end{align}
Moreover, $\psi^N(t, \cdot)$ is 2-Lipschitz, uniformly in $t$, by Lemma \ref{lem: grad}.

\begin{remark}\label{rem:strong}
In \cite{Fla-Iss-Rus-2017} the authors work in $d$ dimensions and find weak solutions for the SDE for $Y$. However, for $d=1$ both equations \eqref{eq: Y} and \eqref{eq: approximated Y} admit a unique \emph{strong} solution if $\gamma:=\delta-1/p>1/2$. That holds because the diffusion coefficient is $\gamma$-H\"older continuous (see \eqref{holder}) and the drift is Lipschitz continuous. This result is used in the proof of Proposition \ref{pr:main2} to justify the use of the same Brownian motion when estimating $Y^N-Y$. 

It then follows that $X_t = \psi(t,Y_t)$ and $X_t^N = \psi^N(t,Y_t^N)$ are adapted to the Brownian filtration and, in that sense, they are `strong' virtual solutions to  \eqref{eq: original sde} and \eqref{eq: approximated sde}, respectively. Moreover, they are unique up to indistinguishability because of the one-to-one mapping between $Y$, $Y^N$ and $X$, $X^N$: for example, if two different solutions $X$ and $\tilde X$ of \eqref{eq: original sde} exist, they give rise to two different solutions $Y$ and $\tilde Y$ of \eqref{eq: Y}, which is impossible by uniqueness of the solution to \eqref{eq: Y}. 
\end{remark}

We conclude this section with some further remarks on the set $\mathcal K(\beta, q)$ and on the different choices $(\delta,p)$, and we explain the implications for the solution $u$. To facilitate the discussion, let us consider $(\beta_0,q_0)$ as in Assumption \ref{ass: main} and let us define
\begin{equation}\label{eq:setH}
\mathcal{H}(\beta_0,q_0):=\{(\delta,p)\in \mathcal K(\beta_0, q_0) \, | \, \delta-1/p>1/2\}.
\end{equation}
The sets $\mathcal{H}(\beta_0,q_0)$ and $\mathcal K(\beta_0, q_0)$ are illustrated in Figure \ref{fig:Kagain}.

\begin{figure}[ht]
	\centering
	\begin{tikzpicture}
	% axes	
	\draw[->] (-0.5,0)--(9,0) node[below right] {$\frac1p$};
	\draw [->] (0,-0.5)--(0,5) node[left] {$\delta$};

	% horizontal lines
	\draw (-.05,4)  node[anchor=east]{$1$}--(.05,4); %1
	\draw[dashed] (-.25,3.75)  node[anchor=east]{$1-\beta_0$}--(8,3.75); %1-\beta horiz
	\draw (-.05,2)  node[anchor=east]{$\frac12$}--(.05,2); %1/2
	\draw (-.05,1)  node[anchor=east]{$\frac14$}--(.05,1); %1/4
	\draw[dashed] (-.25,.25)  node[anchor=east]{$\beta_0$}--(8,.25); %\beta horiz
	
	% vertical lines
	\draw[dashed] (.5,-.25)  node[anchor=north]{$\beta_0$}--(.5, 4); %\beta0 vert
	\draw[dashed] (1,-.25) node[anchor=north]{$ \frac1{q_0}$} --(1,4); % 1/q0
	\draw (2,-.05) node[anchor=north]{$ \frac1{4}$} --(2,0.05); %1/4
	\draw (4,-.05)  node[anchor=north]{$\frac12$}--(4, .05); %1/2
	\draw[dashed] (7.5,-.25) node[anchor=north]{$1-\beta_0=:\frac{1}{\tilde q_0}$} --(7.5,4);%1/tilde q0
	\draw (8,-.05)  node[anchor=north]{$1$}--(8, .05); %1
	
	% large triangle	
	\draw (0, 0) --( 7.5,3.75)	;
	\fill[gray] (1, 0.5)--(1,3.75)--(7.5,3.75)--(1,0.5);

	%black triangle
	\draw (0, 2) --( 4,4) node[anchor=south]{$\delta-\frac1p=\frac12$}	;
	\fill[black]  (1, 2.5) --(1, 3.75) --(3.5, 3.75) --(1, 2.5)  ;
	
	% vertical blue lines 
	\draw[blue, dotted] (2.5,-.65) node[anchor=north]{$ \beta_1$} --(2.5,4); % beta1
	\draw[blue,dotted] (3.25,-.65) node[anchor=north]{$ \frac1{q_1}$} --(3.25,4); % 1/q1
	\draw[blue, dotted] (5.5,-.65) node[anchor=north]{$ 1-\beta_1:=\frac1{\tilde q_1}$} --(5.5,4); % 1/q1

	% horizontal blue lines 
	\draw[blue,dotted] ( -0.65,2.75) node[anchor=east]{$ 1-\beta_1$} --(8, 2.75); %1- beta1
	\draw[blue,dotted] ( -0.65,1.25) node[anchor=east]{$ \beta_1$} --(8, 1.25); %beta1
	
	% light gray triangle
	\fill[mygray]  (3.25, 2.75) --(3.25, 1.625) --(5.5, 2.75)-- (3.25, 2.75)  ;
	
	%names of K and H
	\node at (4.5,1.5) {$\mathcal K(\beta_1, q_1)$};
	\draw (4.5,1.75) arc (20:40:2.5cm);
	\node at (1.5,4.5) {$\mathcal H(\beta_0, q_0)$};
	\draw (1.5,3.5) arc (20:40:2.5cm);
	\node at (6.7,2.5) {$\mathcal K(\beta_0, q_0)$};
	\draw (7,2.75) arc (20:40:2.5cm);
	\end{tikzpicture}

	\caption{The sets $\mathcal K(\beta_0, q_0)$ (large grey triangle),  $\mathcal K(\beta_1, q_1)$ (small light grey triangle) and  $\mathcal H(\beta_0, q_0)$   (small black triangle) are drawn here.}
	\label{fig:Kagain}
\end{figure}
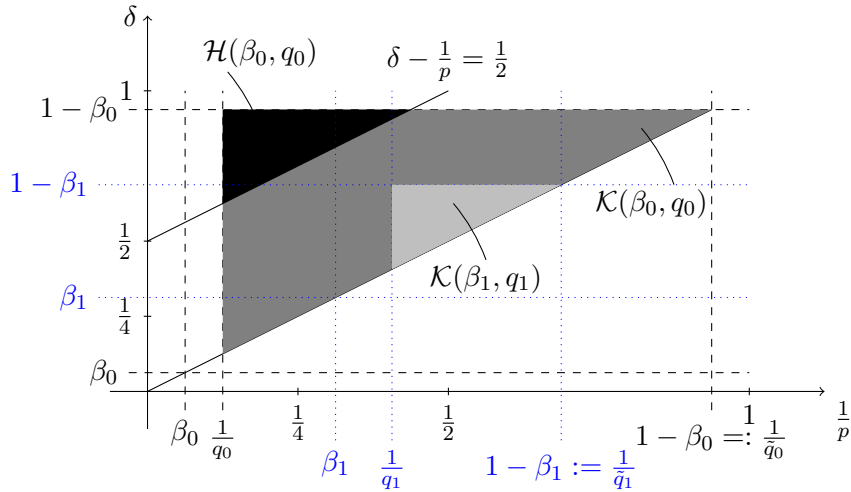

\begin{remark}\label{rem:uniqueness}
Let $b\in \mathcal  C^{\kappa} ([0,T], H^{-\beta_0}_{\tilde q_0, q_0})$, where $(\beta_0,q_0)$ is fixed and satisfies  Assumption \ref{ass: main}. Recall that Assumption \ref{ass:mainFIR} is automatically  satisfied. 
\begin{itemize}
\item [(i)] The set $\mathcal H(\beta_0, q_0)$ defined in \eqref{eq:setH} is not empty (see the black triangle in Figure \ref{fig:Kagain}).
\item [(ii)] Although our drift $b$ satisfies Assumption \ref{ass: main}, in Propositions \ref{pr:main2} and  \ref{prop:Haar} we chose to look at it as an element of a larger space. In particular, we used the embedding (see Remark \ref{rm:b^N}) 
\[
b\in\mathcal{C}([0,T]; H^{-\beta_0}_{q_0,\tilde q_0})\subset \mathcal{C}([0,T]; H^{-\beta_1}_{q_1,\tilde q_1}),
\] 
for $0<\beta_0<\beta_1<\frac12$ and $ q_0>q_1>\tilde q_1>\tilde q_0$.  
\item [(iii)] Since the solution of \eqref{eq: pde} is unique in $\mathcal K(\beta_0,q_0)$ (see large triangle in Figure \ref{fig:Kagain}),  then we can always pick $(\delta,p)\in\mathcal H(\beta_0,q_0)$ such that $u\in \mathcal C([0,T];H^{1+\delta}_p)$ and  $\nabla u(t,\,\cdot\,)\in\mathcal C^\gamma $ for  $\gamma= \delta-\tfrac1p>1/2$ and all $t\in[0,T]$.
The largest H\"older exponent we can find in $\mathcal H(\beta_0,q_0)$ is $\gamma_0-\eps$, where $\gamma_0=1-\beta_0-1/q_0$ and $\eps>0$ is arbitrarily small. This means that for the solution of \eqref{eq: pde} we have $\nabla u(t,\,\cdot\,)\in\mathcal C^{\gamma} $ for all $t\in[0,T]$ and any $\gamma<\gamma_0$.
\end{itemize}
\end{remark}

%%%%%%%%%%%%%%%%%%%%%%%%%%%%%%%%%%%%%%%%%%%%%%%%%%%%%%%%%%%%%%%%%%%

%%%%%%%%%%%%%%%%%%%%%%%%%%%%%%%%%%%%%%%%%%%%%%%%%%%%%%%%%%%%%%%%%%%

\section{Convergence rate of $X^N-X$}\label{sec:thm1}

In this section we prove Proposition \ref{pr:main2}. 
It turns out that in order to show the convergence rate of $X^N$ to $X$ stated in  Proposition \ref{pr:main2} we must provide an upper bound for the local time at zero of $Y-Y^N$. Recall that for any real-valued continuous semi-martingale $\bar Y$, the local time  $L_t^0(\bar Y)$ is  defined as 
\begin{align}\label{eq:LT}
L^0_t(\bar Y)=\lim_{\eps\to 0}\frac{1}{2\eps}\int_0^t\mathds{1}_{\{|\bar Y_s|\le \eps\}}\di \langle \bar Y\rangle_s,\quad \mathbb{P}\text{-a.s.}
\end{align}
for all $t\geq 0$.
Now we derive a bound on \eqref{eq:LT} that will be needed later on. 
\begin{lemma}\label{lm:g}
For any  $\varepsilon\in(0,1)$ and any real-valued, continuous semi-martingale $\bar Y$ we have 
\begin{align}\label{eq:l0}
\E\left[L^0_t(\bar Y)\right]\le& 4\eps-2\E\left[\int_0^t\left(\mathds{1}_{\{\bar Y_s\in(0,\varepsilon)\}}+\mathds{1}_{\{\bar Y_s\ge\varepsilon\}}e^{1-\bar Y_s/\eps}\right)\di \bar Y_s\right]\\
& +\frac{1}{\eps}\E\left[\int_0^t\mathds{1}_{\{\bar Y_s>\eps\}}e^{  1-\bar Y_s/\eps }\di\langle \bar Y\rangle_s\right].\notag
\end{align}
\end{lemma}
\begin{figure}
\centering
\begin{tikzpicture}[scale=0.8, samples=100]
\draw[->] (-1,0) -- (10,0) node[right] {$y$};
\draw[->] (0,-1) -- (0,5) node[left] {};
\node at (0,2)[left] {\footnotesize\color{black} $\varepsilon$};
\node at (0,4)[left] {\footnotesize\color{black} $2\varepsilon$};
\node at (2,0)[below] {\footnotesize\color{black} $\varepsilon$};
\draw[smooth, domain=-1:0, color=blue, very thick] (-1,0) -- (0,0) ;
\draw[smooth, domain = 0:2, color=blue, very thick] (0,0) -- (2,2);
\node at (3,2.5)[below] {\footnotesize\color{blue}$g_\varepsilon(y)$};
\draw[smooth, dashed, domain = 0:10, color=gray]  (0,4) -- (10,4);
\draw[smooth, dashed, domain = 0:2, color=gray] (0,2) -- (2,2);
\draw[smooth, dashed, domain = 0:2, color=gray] (2,0) -- (2,2);
\fill [black] ($(0,2)$) circle (1pt) ;
\fill [black] ($(0,4)$) circle (1pt) ;
\fill [black] ($(2,0)$) circle (1pt) ;
\draw[smooth, domain=2:10, color=blue, very thick] plot (\x,{2*(2-e/(e^(\x/2)) )});
\end{tikzpicture}
\caption{The function $y\mapsto g_\varepsilon (y)$.}
\label{fig:g}
\end{figure}
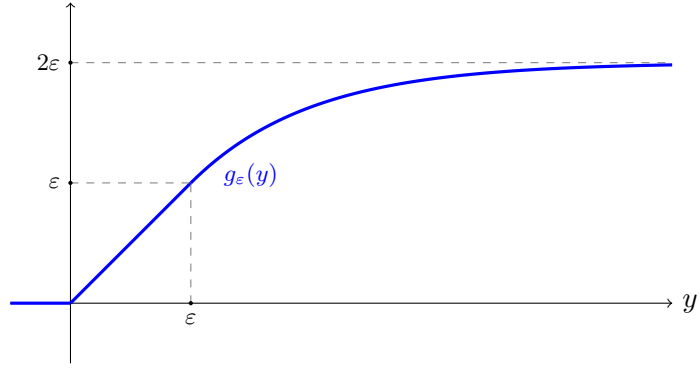
\begin{proof}
For $\eps\in(0,1)$ and $y\in\R$ we define (see Figure \ref{fig:g})
\[
g_\eps(y):=0\cdot\mathds{1}_{\{y<0\}} +y \mathds{1}_{\{0\le y<\eps\}}+\eps\left[2-e^{ 1-y/\eps }\right]\mathds{1}_{\{y\ge \eps\}}.
\]
Straightforward calculations allow to show that $g_\eps\in \mathcal{C}^1(\R\setminus\{0\})$ and it is semi-concave, in the sense that $y\mapsto g_\eps(y)-y^2$ is concave. Moreover, we have 
\begin{align}
\label{eq:g1}&0\le g_\eps(y)\le 2\eps, \quad \text{for }y\in\R\\
&g'_\eps(y)=\mathds{1}_{\{0\le y<\eps\}}+e^{ 1-y/\eps }\mathds{1}_{\{y\ge \eps\}}, \quad \text{for }y\in\R\\
&g''_\eps|_{(-\infty,0)}=g''_\eps|_{(0,\eps)}=0,\\
\label{eq:g4}&g''_\eps(y)=-\eps^{-1}e^{ 1-y/\eps },\quad \text{for }y>\eps.
\end{align}
Now, an application of It\^o-Tanaka formula gives
\begin{align}
&g_\eps(\bar Y_t)-g_\eps(\bar Y_0)\notag\\
&=\int_0^t g'_\eps(\bar Y_s)\mathds{1}_{\{\bar Y_s\neq 0\}}\di\bar Y_s+\frac{1}{2}\int_0^t g''_\eps(\bar Y_s)\mathds{1}_{\{\bar Y_s\neq 0\}\cap\{\bar Y_s\neq\eps\}}\di \langle \bar Y\rangle_s\\
&\:\:\:+\frac{1}{2}[g'_\eps(0+)-g'_\eps(0-)]L^0_t(\bar Y)\notag
\end{align}
where $g'_\eps(0\pm)$ denotes the left/right limit of the derivative at zero. Rearranging terms, taking expectations and using \eqref{eq:g1}--\eqref{eq:g4} gives \eqref{eq:l0}.
\end{proof}
  
The next lemma controls the approximation error between $u$ and $u^N$. We recall that $a^N = P_{\eta_N} b^N$.
\begin{lemma}\label{lem: morrey}
Let Assumption \ref{ass:mainFIR} hold and fix $(\delta,p)\in {\mathcal{K}}(\beta,q)$. Let $u, u^N\in \mathcal{C}([0,T];H_p^{1+\delta})$ be the mild solutions to \eqref{eq: pde} and \eqref{eq:approxpde}, respectively. Then, for $\rho>\rho_0$ and $N>N_0$ as in Remark \ref{rem:lN}, and all $t\in[0,T]$ we have 
    \begin{align}\label{eq:ineq}
    \left\{
    \begin{array}{l}
    \norme{u^N(t) - u(t)}_{L^\infty}\leq \kappa_\rho\norme{b-a^N}_{{\infty},H^{-\beta}_{q}}\\[+10pt]        
    \norme{\nabla u^N(t) - \nabla u(t)}_{L^\infty}\leq \kappa_\rho \norme{b-a^N}_{{\infty},H^{-\beta}_{q}}
    \end{array}
    \right.
    \end{align} 
with 
\begin{equation}\label{eq:kappa}
\kappa_\rho :=c\cdot c(\rho)\cdot e^{\rho T},
\end{equation}
where $c(\rho)>0$ is given in \eqref{eq:crho} and $c>0$.
\end{lemma}    

\begin{proof}
Let $\gamma:=\delta-1/p$. Since $\gamma>0$, by the  fractional Morrey's inequality \eqref{eq:morrey} we have $H^{1+\delta}_p\subset \mathcal{C}^{1,\gamma}$ and we can find $c>0$ such that for all $t\in[0,T]$ it holds
    \begin{equation}\label{eq:M1}
    \begin{cases}
    \norme{u^N(t) - u(t)}_{L^\infty} &\leq\norme{u^N(t) - u(t)}_{\mathcal{C}^{1,\gamma}}\leq c\norme{u^N(t)-u(t)}_{H^{1+\delta}_{p}}\\        
    \norme{\nabla u^N(t) - \nabla u(t)}_{L^\infty}&\leq\norme{u^N(t) - u(t)}_{\mathcal{C}^{1,\gamma}}  \leq c\norme{u^N(t)-u(t)}_{H^{1+\delta}_{p}}.
    \end{cases}        
    \end{equation}        
    Then, recalling \eqref{eq:norm-rho} and \eqref{eq:bconv} we easily obtain
    \begin{equation}\label{eq:M2}
    \norme{u^N-u}_{\infty,H^{1+\delta}_{p}}\leq e^{\rho T} \norme{u^N-u}_{\infty,H^{1+\delta}_{p}}^{(\rho)}\leq c(\rho)e^{\rho T}\norme{b-a^N}_{{\infty},H^{-\beta}_{q}}.
    \end{equation}
Combining \eqref{eq:M1} and \eqref{eq:M2} gives \eqref{eq:ineq}.
\end{proof}

Now we provide a bound on the difference $\psi-\psi^N$, where $\psi$ is defined in \eqref{eq:psi}. 

\begin{lemma}\label{lem: psi}
Take $\rho\!>\!\rho_0$ and $N\!>\!N_0$ as in Remark \ref{rem:lN} and $\kappa_\rho$ as in Lemma \ref{lem: morrey}. Under Assumption \ref{ass:mainFIR} we have
\begin{equation}\label{eq:psineq}
    \sup_{(t,x)\in[0,T]\times\mathbb{R}}\ \left|\psi\left(t,x\right)-\psi^N\left(t,x\right)\right|\leq 2\kappa_\rho \norme{b-a^N}_{{\infty},H^{-\beta}_{q}}.
    \end{equation} 
\end{lemma}
\begin{proof}
Recall that $\varphi\in\mathcal C([0,T]; \mathcal C^1)$ was defined in \eqref{eq:phi}. For any $y,y'\in\R$ we have
\begin{align*}
&|\varphi(t,y)-\varphi(t,y')|^2\\
&\ge |u(t,y)-u(t,y')|^2+|y-y'|^2-2|u(t,y)-u(t,y')||y-y'|\\
&=\left(|y-y'|-|u(t,y)-u(t,y')|\right)^2\ge \tfrac{1}{4}|y-y'|^2
\end{align*}
where the final inequality uses  Lemma \ref{lem: grad}. 
Now, taking $y=\psi(t,x)$ and $y'=\psi^N(t,x)$ in the above inequality gives
    \begin{align*}
    \left|\varphi\left(t,\psi\left(t,x\right)\right)-\varphi\left(t,\psi^N\left(t,x\right)\right)\right| 
    &\geq \frac{1}{2} \left|\psi\left(t,x\right)-\psi^N\left(t,x\right)\right|.
    \end{align*}    
    The latter implies
    \begin{align*}
    \left|\psi\left(t,x\right)-\psi^N\left(t,x\right)\right|&\leq 2\ \left|\varphi\left(t,\psi\left(t,x\right)\right)-\varphi\left(t,\psi^N\left(t,x\right)\right)\right|
    \\ &
    = 2\ \left|\varphi^N\left(t,\psi^N\left(t,x\right)\right)-\varphi\left(t,\psi^N\left(t,x\right)\right)\right|
    \end{align*} 
where the final equality uses $\varphi(t,\psi(t,x))=x=\varphi^N(t,\psi^N(t,x))$. By definition of $\varphi$ and $\varphi^N$ and \eqref{eq:ineq}, from Lemma \ref{lem: morrey} we also obtain 
\begin{align*}
&\left|\varphi^N\left(t,\psi^N\left(t,x\right)\right)-\varphi\left(t,\psi^N\left(t,x\right)\right)\right|\\
&=\left|u^N\left(t,\psi^N\left(t,x\right)\right)-u\left(t,\psi^N\left(t,x\right)\right)\right|\\
&\le \norme{u(t)-u^N(t)}_{L^\infty}\le \kappa_\rho\|b-a^N\|_{{\infty},H^{-\beta}_{ q}}.
\end{align*}    
Combining the above expressions we get \eqref{eq:psineq}.    
\end{proof}
 
In the next proposition we provide an upper bound for the local time at zero of $Y-Y^N$. For $r>0$ we denote by $o(r)$ a generic function with $o(r)/r\to 0$ as $r\to 0$.
\begin{proposition}\label{prop:localtime}
Let Assumption \ref{ass: main} and \ref{ass: b^N} hold. Take arbitrary $(\beta,q)$ that satisfy Assumption \ref{ass:mainFIR} and such that $\beta\in[\beta_0,\tfrac{1}{2})$ and  $q_0\geq q > \tilde q\ge \tilde q_0$. Then, for any $1/2<\gamma<\gamma_0$, we have
\begin{align}\label{eq:localtime}
\E\left[L^0_T(Y^N-Y)\right]\le& c_\gamma\|a^N-b\|^{2\gamma-1}_{\infty,H^{-\beta}_{q}}+c'\E\left[\int_0^T|Y^N_s-Y_s|\di s\right]\\
& + o( \|a^N-b\|^{2\gamma-1}_{\infty,H^{-\beta}_{q}})\notag
\end{align}
as $N\to\infty$, with $c_\gamma,c'>0$ given constants ($c_\gamma$ depending on $\gamma$). 
\end{proposition}

\begin{proof}
It is clear that $\mathcal H(\beta_0,q_0)\neq\emptyset$ by Remark \ref{rem:uniqueness} and that we can choose $(\beta,q)$ as indicated, thanks to Assumption \ref{ass: main}. Let us also recall Remark \ref{rem:uniqueness}, part (iii), so that for the solution $u$ of \eqref{eq: pde} we have $\nabla u(t,\,\cdot\,)\in\mathcal C^{\gamma}$ for all $1/2<\gamma<\gamma_0$. 

Thanks to \eqref{eq: Y} and \eqref{eq: approximated Y} it is easy to derive the dynamics of $\bar Y:=Y^N-Y$ (recall that $Y$ and $Y^N$ are strong solutions by Remark \ref{rem:strong}). Then, applying Lemma \ref{lm:g} we obtain
\begin{align}\label{eq:e0}
&\E\left[L^0_t(Y^N - Y)\right]\le 4\eps \\
&-\!2(1\!+\!\lambda)\E\!\bigg[\int_0^t\!\! \left(\!\mathds{1}_{\{\bar Y_s\in(0,\varepsilon)\}}\!+\! \mathds{1}_{\{\bar Y_s\ge\varepsilon\}}e^{1-\bar Y_s/\eps}\right)\!\left(u^N(s,\psi^N(s,Y^N_s))\!-\!u(s,\psi(s,Y_s))\right)\di s\bigg]\notag\\
&+\frac{1}{\eps}\E\left[\int_0^t\!\!\mathds{1}_{\{ Y^N_s-Y_s >\eps\}}e^{1-(Y^N_s-Y_s)/\eps}
\!\left(\nabla u^N(s,\psi^N(s,Y^N_s))\!-\!\nabla u(s,\psi(s,Y_s))\right)^2\di s\right]\notag
\end{align}
where we have removed the martingale term. 
Adding and subtracting terms we have
\begin{align*}
&\big|u^N(s,\psi^N(s,Y^N_s))\!-\!u(s,\psi(s,Y_s))\big|\\
&\le \big|u^N(s,\psi^N(s,Y^N_s))\!-\!u(s,\psi^N(s,Y^N_s))\big|+\big|u(s,\psi^N(s,Y^N_s))\!-\!u(s,\psi(s,Y^N_s))\big|\\
&\quad+\big|u(s,\psi(s,Y^N_s))\!-\!u(s,\psi(s,Y_s))\big|.
\end{align*}
In order to estimate the right-hand side of the expression above we use Lemma \ref{lem: grad},  Lemma \ref{lem: morrey} and Lemma \ref{lem: psi}, upon recalling that $a^N\to b$ in the space $\mathcal{C}([0,T];H^{-\beta}_{q})$, by Remark \ref{rm:b^N} part (ii). Since $\psi(s,\cdot)$ and $\psi^N(s,\cdot)$ are 2-Lipschitz, 
and $u(s,\cdot)$ and $u^N(s,\cdot)$ are $\tfrac{1}{2}$-Lipschitz, uniformly in $s\in[0,T]$ we have 
\begin{align}\label{eq:e1}
&\left|u^N(s,\psi^N(s, Y^N_s))-u(s,\psi(s, Y_s))\right|\le 2 \kappa_\rho \|b-a^N\|_{{\infty},H^{-\beta}_{ q}} + |Y^N_s-Y_s|,
\end{align}
for $\rho\!>\!\rho_0$ and $N\!>\!N_0$ as in Remark \ref{rem:lN} and $\kappa_\rho$ as in Lemma \ref{lem: morrey}.
Similarly, for the term in \eqref{eq:e0} involving the gradient of $u$ and $u^N$ we get 
\begin{align}\label{eq:e2}
\big|\nabla u^N&(s,\psi^N(s, Y^N_s ))-\nabla u(s,\psi(s, Y_s))\big| \\
\le &\kappa_\rho \|b-a^N\|_{{\infty},H^{-\beta}_{ q}} +  2^{\gamma+1}  \kappa_\rho^{\gamma} \| u\|_{\infty,\mathcal{C}^{1,\gamma}} \|b-a^N\|^{\gamma}_{{\infty},H^{-\beta}_{q}} \notag\\ 
& + \left|\nabla u(s,\psi(s, Y^N_s))-\nabla u(s,\psi(s, Y_s )) \right|,\notag
\end{align} 
where for the second term on the right-hand side above we used that $\nabla u(t,\,\cdot\,)$ is $\gamma$-H\"older continuous for any $1/2<\gamma<\gamma_0$, and then used Lemma \ref{lem: psi}. Now, plugging \eqref{eq:e1} and \eqref{eq:e2} into \eqref{eq:e0} and using the well-known inequality 
\[
(x_1+x_2+\ldots+x_k)^2\le k(x_1^2+\ldots+x^2_k )
\]
we obtain
\begin{align}\label{eq:l1}
\E&\left[L^0_t(Y^N - Y)\right]\\
\le& 4\eps + 4(1\!+\!\lambda)\left( 2\kappa_\rho t\|a^N-b\|_{{\infty},H^{-\beta}_{q}} + \E\!\left[\int_0^t|Y^N_s-Y_s|\di s\right]\right) \notag\\
&+\frac{1}{\eps}3t\|a^N-b\|_{{\infty},H^{-\beta}_{q}}\left(\kappa^2_\rho\|a^N-b\|_{{\infty},H^{-\beta}_{q}}+4(2\kappa_\rho)^{2\gamma} \|u\|^2_{\infty,\mathcal{C}^{1,\gamma}}\|a^N-b\|^{2\gamma-1}_{{\infty},H^{-\beta}_{q}}\right)\notag\\
&+\frac{1}{\eps} 3\E\left[\int_0^t\!\!\mathds{1}_{\{ Y^N_s-Y_s >\eps\}}e^{1-(Y^N_s-Y_s)/\eps} \left|\nabla u(s,\psi(s,Y^N_s))-\nabla u(s,\psi(s,Y_s))\right|^2 \di s\right]. \notag
\end{align}

For simplicity, we denote by $I^{N,\eps}_t$ the last term in \eqref{eq:l1}. To find an upper bound for $I^{N,\eps}$ we pick $\zeta\in(0,1)$ such that $\gamma\zeta>1/2$ and recall that $\eps\in(0,1)$ so that $\eps^\zeta>\eps$. Using the fact that $\nabla u(s,\cdot)$ is $\gamma$-H\"older continuous uniformly in $s\in[0,T]$ with constant $\|u\|_{\infty,\mathcal{C}^{1,\gamma}}$, that $\psi$ is 2-Lipschitz and that $\nabla u$ is uniformly bounded by $1/2$ thanks to Lemma \ref{lem: grad}, we get
\begin{align}\label{eq:l2}
I^{N,\eps}_t\le& \frac{1}{\eps} 3 \E\left[\int_0^t\!\!\mathds{1}_{\{ \eps< Y^N_s-Y_s \le \eps^\zeta\}}e^{1-(Y^N_s-Y_s)/\eps}    2^{2\gamma} \|u\|^2_{\infty,\mathcal{C}^{1,\gamma}} |Y^N_s\!-\!Y_s|^{2\gamma} \di s\right]\\
&+\frac{1}{\eps} 3 \E\left[\int_0^t\!\!\mathds{1}_{\{ Y^N_s-Y_s > \eps^\zeta\}}e^{1-(Y^N_s-Y_s)/\eps} (2\|\nabla u\|_{L^\infty})^2\di s\right]\notag\\
\le & \frac{1}\eps \left( 3 T 2^{2\gamma}\|u\|^2_{\infty,\mathcal{C}^{1,\gamma}}
\eps^{2\gamma\zeta}+ 3T e^{1-\eps^{\zeta-1}} \right)\notag.
\end{align} 
With no loss of generality we can take $\eps=\|a^N-b\|_{{\infty},H^{-\beta}_{q}}$. Combining \eqref{eq:l1} and \eqref{eq:l2} we then find
\begin{align*}
\E&\left[L^0_t(Y^N - Y)\right]\\
\le& c_1 \|a^N-b\|_{{\infty},H^{-\beta}_{  q}}+c_2\|a^N-b\|^{2\gamma-1}_{{\infty},H^{-\beta}_{ q}}+c_3\|a^N-b\|^{2\gamma\zeta-1}_{{\infty},H^{-\beta}_{ q}}\notag\\
&+c_4\|a^N-b\|^{-1}_{{\infty},H^{-\beta}_{ q}}\exp\left(1-\|a^N-b\|^{\zeta-1}_{{\infty},H^{-\beta}_{  q}}\right)\notag\\
&+c_5\E\left[\int_0^t|Y^N_s-Y_s|\di s\right],\notag
\end{align*}
where
\begin{align*}
&c_1=4\!+\!8T(1+\lambda)\kappa_\rho + 3T\kappa^2_\rho\,,\qquad c_2=12 T (2\kappa_\rho)^{2\gamma} \|u\|^2_{\infty,\mathcal{C}^{1,\gamma}}\,,\\
&c_3=3 T  2^{2\gamma}\|u\|^2_{\infty,\mathcal{C}^{1,\gamma}}\,,\qquad c_4=3T\,,\qquad c_5=4(1+\lambda) .
\end{align*}
Since $\zeta\in(0,1)$, the term containing the exponential goes to zero faster than any polynomial as $N\to\infty$. Moreover, $\zeta$ can be taken arbitrarily close to one and $1/2<\gamma<\gamma_0$ was also arbitrary, hence \eqref{eq:localtime} holds and the proof is complete with $c_\gamma=c_3$ and $c'=c_5$.
\end{proof}

We are now ready to prove Proposition \ref{pr:main2}, which we recall below for the reader's convenience.

\noindent\textbf{Proposition \ref{pr:main2}.}
{\em Let Assumptions \ref{ass: main} and \ref{ass: b^N} hold. 
Take any $(\beta,q)$ such that $\beta\in(\beta_0,\tfrac{1}{2})$ and $q_0\ge q>\tilde q\ge \tilde q_0$, where $\tilde q:=(1-\beta)^{-1}$. Then, for any $1/2<\gamma<\gamma_0$ there is a constant $C_{\gamma}>0$ such that 
\begin{align}\label{eq:main2b}
\sup_{0\le t \le T}\E&\left[|X^N_t\!-\!X_t|\right]\le C_{\gamma} \|P_{\eta_N}\,b^N-b\|^{2\gamma-1}_{\infty,H^{-\beta}_q}  
\end{align}
as $N\to\infty$.}

\begin{proof} 
Recall the embedding $b\in \mathcal C([0,T]; H^{-\beta_0}_{\tilde q_0, q_0}) \subset \mathcal C([0,T]; H^{-\beta}_{\tilde q, q})$ (Remark \ref{rm:b^N} part (i)) and notice that $a^N=P_{\eta_N}b^N\to b$ in $\mathcal{C}([0,T];H^{-\beta}_{ q})$ as $N\to\infty$ (Remark \ref{rm:b^N} part (ii)).
 Next we note that
\begin{align*}
|X^N_t-X_t|=&|\psi^N(t,Y^N_t)-\psi(t,Y_t)|\\
\le& |\psi^N(t,Y^N_t)-\psi(t,Y^N_t)|+|\psi(t,Y^N_t)-\psi(t,Y_t)|\\
\le&2\kappa_\rho\|b-a^N\|_{{\infty},H^{-\beta}_{ q}}+2 |Y^N_t-Y_t| ,
\end{align*}
where in the final inequality we have used Lemma \ref{lem: psi} and that $\psi(t,\cdot)$ is 2-Lipschitz, uniformly in $t\in[0,T]$ (Lemma \ref{lem: grad}). Therefore it is sufficient to find a bound for $|Y^N_t-Y_t|$.

From It\^o-Tanaka formula we get
\begin{align*}
|Y^N_t-Y_t|=|y^N_0-y_0|+\int_0^t\text{sign}(Y^N_s-Y_s)\di (Y^N_s-Y_s)+\frac{1}{2}L^0_t(Y^N-Y).
\end{align*}
Taking expectation, using \eqref{eq: Y} and \eqref{eq: approximated Y} and removing the martingale term we obtain
\begin{align}\label{eq:35}
\E&\left[|Y^N_t\!-\!Y_t|\right]\\
=&|u^N(0,x)-u(0,x)|+  \frac12 \E\left[L^0_t(Y^N-Y)\right]\notag\\
&+\!(1\!+\!\lambda)\E\left[\int_0^t\!\text{sign}(Y^N_s\!-\!Y_s)\left(u^N(s,\psi^N(s,Y^N_s))\!-\!u^N(s,\psi(s,Y_s)\right)\di s\right]\notag\\
\notag \le &|u^N(0,x)-u(0,x)|+\frac{1}{2}\E\left[L^0_t(Y^N-Y)\right]\\
&+ 2(1\!+\!\lambda)t \kappa_\rho  \|b-a^N\|_{{\infty},H^{-\beta}_{q}} +(1\!+\!\lambda) \E\left[\int_0^t\!|Y^N_s-Y_s|\di s\right],\notag  
\end{align}
where the inequality follows from the bound \eqref{eq:e1} used in Proposition \ref{prop:localtime}. Using \eqref{eq:ineq} we have 
\[
|u^N(0,x)-u(0,x)|\le \kappa_\rho \|a^N-b\|_{{\infty},H^{-\beta}_{q}}. 
\]
Thanks to Proposition \ref{prop:localtime}, we have an upper bound for the local time and, in particular, for any $1/2<\gamma<\gamma_0$ we have
\begin{align}\label{eq:36}
\E&\left[|Y^N_t\!-\!Y_t|\right]\le C\|a^N-b\|^{2\gamma-1}_{{\infty},H^{-\beta}_{q}}+C'\E\left[\int_0^t|Y^N_s-Y_s|\di s\right]+ o\big(\|a^N-b\|^{2\gamma-1}_{{\infty},H^{-\beta}_{q}} \big),
\end{align}
as $N\to\infty$, where
\[
C=\tfrac{3}{2}2^{2\gamma}T\|u\|^2_{\infty,\mathcal{C}^{1,\gamma}},\quad C'=3(1+\lambda).
\]
By an application of Gronwall's lemma we conclude the proof and the constant $C_{\gamma}>0$ in \eqref{eq:main2b} can be taken as
$C_{\gamma}= (1+C) e^{TC'}$ since $o\big(\|a^N-b\|^{2\gamma-1}_{{\infty},H^{-\beta}_{q}} \big)\leq \|a^N-b\|^{2\gamma-1}_{{\infty},H^{-\beta}_{q}}$ for $N$ large enough.
\end{proof}

%%%%%%%%%%%%%%%%%%%%%%%%%%%%%%%%%%%%%%%%%%%%%%%%%%%%%%%%%%%%%%%%%%%%%%%%%%%%%%%%%%

\section{Convergence rate for Euler-Maruyama scheme}\label{sec:thm2} 
In this section we prove Proposition \ref{thm:EM}, which gives a bound for 
\begin{align}\label{eq:diff}
\E\left[|X^{N,m}_t-X^N_t|\right],
\end{align}
where $N$ is fixed and $m$ tends to infinity. Let us start with some initial considerations.

For each $N$ the drift in the equation for $X^N$ is certainly bounded, $\tfrac{1}{2}$-H\"older continuous in time and Lipschitz continuous in space, since $b^N$ is bounded measurable. Therefore we could use classical results (see, e.g., \cite{kloeden}) on the convergence of the Euler-Maruyama scheme in order to obtain a rate of convergence of order $m^{-1/2}$. Following this approach, the multiplicative constant in front of the rate $m^{-1/2}$ depends exponentially on the Lipschitz constant of the drift, i.e., in our case on the exponential of $\|\nabla (P_{\eta_N} b^N)\|_{\infty,L^\infty}$.

Substantially more refined results for SDEs with additive noise were obtained recently by \cite{dareiotis2020}, who found a strong convergence rate of order $m^{-1/2}$ under the sole requirement of a bounded (time-homogeneous) drift; see \cite[Theorem 1.2]{dareiotis2020}. In that theorem the multiplicative constant that appears in front of the rate $m^{-1/2}$ depends exponentially on the $L^\infty$-norm of the drift (\cite[Lemma 2.2]{dareiotis2020}), i.e., in our case on the exponential of $\|P_{\eta_N} b^N\|_{\infty,L^\infty}$.

Both $\|\nabla (P_{\eta_N} b^N)\|_{\infty,L^\infty}$ and $\|P_{\eta_N} b^N\|_{\infty,L^\infty}$ explode as $N\to\infty$ (and $\eta_N\to0$) at a rate depending on the inverse of $\eta_N$, hence increasing the overall approximation error exponentially as we send $(N,m)\to\infty$ at the same time. Of course one could let $\eta_N\to 0$ very slowly, in order to compensate for the exponential explosion, but this would produce a very slow rate of convergence of $X^N\to X$ (see Proposition \ref{pr:main2}) hence deteriorating, once again, the overall convergence rate. Here we find a compromise by contenting ourselves with a convergence rate for the Euler Maruyama scheme of order $m^{-1/2}$  (see \eqref{eq:EM})
but with a multiplicative constant (see \eqref{eq:C23}) which grows polynomially with the inverse of $\eta_N$ (rather than exponentially). 

We use a transformation which is the analogue of the one used to define the virtual solutions. That is, we transform the processes $ X^N $ and $ X^{N,m}$ into new processes $Y^N $ and $ Y^{N,m}$ whose dynamics are expressed in terms of It\^o's diffusions with `nice' coefficients.

Throughout this section Assumptions \ref{ass: main} and \ref{ass: b^N} are enforced.
Since the index $N$ is fixed, it is convenient to simplify the notation and write
\[
\XX:=X^N,\quad\XX^m:=X^{N,m},\quad\aa:=a^N=P_{\eta_N}b^N,\quad \hat u:=u^N,
\]
and denote 
\begin{align*}
&\hat\varphi(t,x):=\varphi^N(t,x)  = x + \hat  u (t,x)\quad\text{and}\quad\hat \psi (t,x):=\psi^N (t,x). 
\end{align*}
Using this notation we can define $\YY_t := \hat \varphi (t, \XX_t)$ and $\YY^m_t :=\hat \varphi (t, \XX^m_t)$ so that $\XX_t  = \hat \psi (t, \YY_t)$ and $\XX^m=\hat \psi (t, \YY^m_t)$. Recalling that $\hat \psi(t, \cdot)$  is 2-Lipschitz, uniformly in $t$, we obtain  
\begin{equation}\label{eq:X<Y}
\E\left[|\XX^{m}_t-\XX_t|\right]\leq 2 \E\left[|\YY^{m}_t-\YY_t|\right].
\end{equation}
In order to estimate the right-hand side in the expression above, we first find the dynamics of $\YY$ and $\YY^m$ in the next lemma. 

\begin{lemma}\label{lem:virt}
The dynamics of $\YY$ is given by
\begin{align}\label{eq:YY1}
\YY_t=\hat \varphi(0,x) \!+\!(1\!+\!\lambda)\int_0^t\! \hat u(s,\hat\psi(s,\YY_s))\di s\!+\!\int_0^t\!(\nabla \hat u(s,\hat \psi(s,\YY_s))\!+\!1)\di W_s,
\end{align}
for all $t\in[0,T]$. Moreover, for any $m\ge 1$   the dynamics of $\YY^m$  is given by
\begin{align}\label{eq:YY2}
\YY^m_t=\hat \varphi(0,x) \! +\!(1\!+\!\lambda)\int_0^t\! \hat u(s,\hat\psi(s,\YY^m_s))\di s\!+\!\int_0^t\!(\nabla \hat u(s,\hat \psi(s,\YY^m_s))\!+\!1)\di W_s\!+\!  E^m_t,
\end{align}
 for all $t\in[0,T]$, where the `{\em error process}' $E^m$ can be written in terms of $\XX^m$ as 
\begin{align*}
 E^m_t:= \int_0^t\left( \aa(t_{k(s)}, \XX^m_{t_{k(s)}})- \aa(s, \XX^m_{s})\right) \left(1+ \nabla\hat u (s,  \XX^m_{s} ) \right)\di s.
\end{align*}
\end{lemma}

\begin{proof}
We start by proving \eqref{eq:YY1}. Since $\aa\in \C^1(\mathbb R)$, the unique mild solution $\hat u$ of the associated PDE \eqref{eq:approxpde} must be a classical solution, i.e., $\hat u \in \mathcal C^{1,2}([0,T)\times\mathbb{R})$. Then, applying It\^o's formula we obtain
\begin{align*}%\label{eq:virt1}
\hat u(t,\XX_t)=&\hat u(0,x)+\int_0^t\left(\partial_t \hat u+\tfrac{1}{2}\Delta \hat u+\aa\nabla \hat u\right)(s,\XX_s)\di s+\int_0^t\nabla \hat u(s,\XX_s)\di W_s\notag\\
=&\hat u(0,x)+\int_0^t\left((1+\lambda) \hat u(s,\XX_s)-\aa(s,\XX_s)\right)\di s+\int_0^t\nabla \hat u(s,\XX_s)\di W_s\notag.
\end{align*}
Plugging this into the definition of $\YY$ and using the SDE \eqref{eq: approximated sde} for $\XX$ we get
\begin{align*}
 \YY_t =& \hat \varphi(t,\XX_t)\\
 =& \hat u(t,\XX_t)+\XX_t\\
=&\hat\varphi(0,x) \!+\!(1\!+\!\lambda)\int_0^t\! \hat u(s,\hat\psi(s,\YY_s))\di s\!+\!\int_0^t\!(\nabla \hat u(s,\hat \psi(s,\YY_s))\!+\!1)\di W_s,
\end{align*}
upon recalling that $\XX_s=\hat\psi(s,\YY_s)$.

The proof of \eqref{eq:YY2} follows the same ideas but we have the additional error term $E^m_t$, due to the special drift of $\XX^m$ in \eqref{eq:EuM}. By  It\^o's formula indeed we obtain
\begin{align}\label{eq:u0}
 \hat u(t,\XX^m_t)=&\hat u(0,x)+ \int_0^t \Big[\left(\partial_t \hat u+\tfrac{1}{2}\Delta \hat u\right)(s,\XX^m_s) +  \aa(t_{k(s)},\XX^m_{t_{k(s)}})\nabla \hat u(s,\XX^m_s)\Big]\di s \\
 &+  \int_0^t\nabla \hat u(s,\XX^m_s)\di W_s.\notag
\end{align}
Using the PDE \eqref{eq:approxpde} we can substitute 
\[
\left(\partial_t \hat u+\tfrac{1}{2}\Delta \hat u\right)(s,\XX^m_s)=\big[(1+\lambda)\hat u-\aa\nabla\hat u-\aa](s,\XX^m_s)
\]
in \eqref{eq:u0} to obtain 
\begin{align}\label{etilde2}
 \hat u(t,\XX^m_t)=&\hat u(0,x)+ \int_0^t \left( (1+\lambda)\hat u(s,\XX^m_s)  - \aa(s,\XX^m_s)\right)\di s\\
 &+  \int_0^t\nabla \hat u(s,\XX^m_s)\di W_s + E^{1,m}_t,\notag
\end{align}
with  
\begin{align}\label{etilde1}
E^{1, m}_t :  =  \int_0^t \left(  \aa(t_{k(s)},\XX^m_{ t_{k(s)} }) - \aa(s,\XX^m_{s}) \right) \nabla \hat u(s,\XX^m_s)\di s.
\end{align}
Using the SDE \eqref{eq: approximated sde} for $\XX^m$ and the definition of $\YY^m$ we get
\begin{align*}
 \YY_t^m  =&\hat\varphi(0,x) \!+\!(1\!+\!\lambda)\int_0^t\! \hat u(s,\hat\psi(s,\YY^m_s))\di s\!+\!\int_0^t\!(\nabla \hat u(s,\hat \psi(s,\YY^m_s))\!+\!1)\di W_s\\
&+ E^{1, m}_t +E^{2, m}_t 
\end{align*}
where
\[
E^{2, m}_t :  =  \int_0^t \left(  \aa(t_{k(s)},\XX^m_{ t_{k(s)} }) - \aa(s,\XX^m_{s}) \right)   \di s.
\]
Hence, \eqref{eq:YY2} follows by setting $E^{m}_t= E^{1, m}_t +E^{2, m}_t $.
\end{proof}
To find a bound for \eqref{eq:X<Y} and prove the rate of convergence of the scheme, we will proceed similarly to the proof of  Proposition \ref{pr:main2}. Indeed, we will apply It\^o-Tanaka formula to $|\YY^m- \YY|$ and estimate the resulting terms. Preliminary bounds are obtained in the next two lemmas.  
\begin{lemma}\label{lem:E}
 Let $s\mapsto|\di E^m_s|$ be the infinitesimal variation of $E^m$. Then 
\begin{align}\label{eq:boundE}
\E\left[\int_0^t|\di E^m_s|\right]\le c_1(N)m^{-1}+c_2(N)m^{-\frac{1}{2}},
\end{align}
with 
\begin{align*}
&c_1(N)=6T\|a^N\|_{\infty,L^\infty} +\frac{3}{4}T^2\|\nabla a^N\|_{\infty,L^\infty}\|a^N\|_{\infty,L^\infty} \quad\text{and}\\
&c_2(N)=T^{\frac{3}{2}}\|\nabla a^N\|_{\infty,L^\infty}+ \frac{3}{2} [a^N]_{\frac12,L^\infty} T^{\frac{3}{2}}.
\end{align*}
\end{lemma}
\begin{proof}
By definition of $E^m_t$ and using $|\nabla \hat u (s, \hat X^m_s)| \leq \tfrac12 $ from Lemma \ref{lem: grad} we have
\begin{align*}%\label{eq:Em-0}
\E&\left[\int_0^t|\di E^{m}_s|\right] \\
\leq &\frac32\E \left[ \int_0^t \left|  \aa(t_{k(s)},\XX^m_{ t_{k(s)} }) - \aa(s,\XX^m_{s}) \right|   \di s \right]\notag\\
\notag
\le& \frac32 \E \Big[ \int_{0}^{t_{1}} \left|  \aa(0,x) - \aa(s,\XX^m_{s}) \right|  \di s + \sum^{k(t)-1}_{k=1}\int_{t_k}^{t_{k+1}} \left|  \aa(t_{k},\XX^m_{ t_{k} }) - \aa(t_{k},\XX^m_{s}) \right|  \di s\Big]\\
\notag
&+ \frac32 \E \Big[ \sum^{k(t)-1}_{k=1}\int_{t_k}^{t_{k+1}} \left|  \aa(t_{k},\XX^m_{s}) - \aa(s,\XX^m_{s}) \right|  \di s+\int_{t_{k(t)}}^{t} \left|  \aa(t_{k(s)},\XX^m_{ t_{k(s)} } ) - \aa(s,\XX^m_{s}) \right|  \di s\Big].
\end{align*}
The first and last term in the final expression above are bounded by $6\frac Tm \|\aa\|_{\infty,L^\infty}$, recalling that $0\le t-t_{k(t)}\le T/m$. In the third term we recall the H\"older seminorm \eqref{eq:Hsemi} and obtain
\begin{align*}
&\E \left[ \sum^{k(t)-1}_{k=1}\int_{t_k}^{t_{k+1}} \left|  \aa(t_{k},\XX^m_{s}) - \aa(s,\XX^m_{s}) \right|  \di s\right]\le \sum^{k(t)-1}_{k=1}\int_{t_k}^{t_{k+1}}\E \left[  \left\| \aa(t_{k}) - \aa(s) \right\|_{L^\infty} \right]\di s\\
&\le [\hat a]_{\frac{1}{2},L^\infty}\sum^{k(t)-1}_{k=1}\int_{t_k}^{t_{k+1}}(s-t_k)^{\frac{1}{2}}\di s\le [\hat a]_{\frac{1}{2},L^\infty}T^{\frac{3}{2}}m^{-\frac{1}{2}},
\end{align*}
where in the final inequality we use that there are at most $m$ terms in the sum.

Finally, for the second term we have 
\begin{align*}
\E \left[ \sum^{k(t)-1}_{k=1}\int_{t_k}^{t_{k+1}} \left|  \aa(t_{k},\XX^m_{t_k}) - \aa(t_k,\XX^m_{s}) \right|  \di s\right]\le \|\nabla \aa\|_{\infty,L^\infty} \sum^{k(t)-1}_{k=1} \int_{t_k}^{t_{k+1}} \E \left[ \left| \XX^m_{ t_{k} } - \XX^m_{s} \right|\right]  \di s.
\end{align*}
Thanks to \eqref{eq:EuM}, each term in the sum above can be easily estimated as
\begin{align*}
&\int_{t_k}^{t_{k+1}} \E \left[ \left| \XX^m_{ t_{k} } - \XX^m_{s} \right|\right]  \di s\\
&=\int_{t_k}^{t_{k+1}} \E \left[ \left| \aa(t_{k},\XX^m_{ t_{k} })(s-t_k) + (W_{s}-W_{t_k}) \right|\right]  \di s\\
&\le \int_{t_k}^{t_{k+1}} \Big(\|\aa\|_{\infty,L^\infty}(s-t_k)+\E \left[ \left|(W_{s}-W_{t_k}) \right|\right] \Big) \di s\\
&= \frac{1}{2}\|\aa\|_{\infty,L^\infty}(t_{k+1}-t_k)^2+\E \left[ \left|W_{1}\right|\right]\int_{t_k}^{t_{k+1}} (s-t_k)^{\frac{1}{2}}  \di s\\
&= \frac{1}{2}\|\aa\|_{\infty,L^\infty}\Big(\frac{T}{m}\Big)^2+\frac{2}{3}\E \left[ \left|W_{1}\right|\right]\Big(\frac{T}{m}\Big)^{\frac{3}{2}}, 
\end{align*}
where again we used $t_{k+1}-t_k=T/m$. 

Combining all of the above estimates we conclude 
\begin{align*}%\label{eq:Em-0}
\E&\left[\int_0^t|\di E^{m}_s|\right]\\
\le&6\|\aa\|_{\infty,L^\infty}\frac{T}{m}+\frac{3}{2}\|\nabla \aa\|_{\infty,L^\infty}\Big(\frac{1}{2}\|\aa\|_{\infty,L^\infty} T^2 m^{-1}+\frac{2}{3}T^{\frac{3}{2}}m^{-\frac{1}{2}}\Big)+ \frac{3}{2} [\hat a]_{\frac12,L^\infty} T^{\frac{3}{2}}m^{-\frac{1}{2}}. 
\end{align*}
Rearranging terms gives \eqref{eq:boundE}.
\end{proof}

\begin{lemma}\label{lem:localtime}
The following holds  
\begin{align}\label{eq:ltYY}
\E\left[L^0_T(\YY-\YY^m)\right]\le& 
4(1+\lambda)
 \E\left[\int_0^T|\YY^m_s-\YY_s|\di s\right]+2\E\left[\int_0^T|\di E^m_s|\right] .
\end{align}
\end{lemma}
\begin{proof}
This estimate uses arguments analogous to those in the proof of Proposition \ref{prop:localtime}.
By Lemma \ref{lm:g} and using the dynamics \eqref{eq:YY1} and \eqref{eq:YY2}, and the fact that $\hat u (t,\cdot)$ and $\hat \psi(t,\cdot)$ are Lipschitz (uniformly in $t\in[0,T]$) we obtain
\begin{align}\label{eq:lt}
\E &\left[L^0_T(\YY-\YY^m)\right] \\
\le &4\eps\!+\! 4(1\!+\!\lambda) \E\left[\int_0^T\left|\YY_s\!-\!\YY^m_s\right|\di s\right]+2\E\left[\int_0^T|\di E^{m}_s|\right]\notag \\
&+ \frac1\eps \E\left[\int_0^T \mathds{1}_{\{\YY_s\!-\!\YY^m_s>\eps\}}e^{(1-(\YY_s\!-\!\YY^m_s)/\eps)}  \left| \nabla \hat u (s, \hat \psi(s, \YY_s) )  -  \nabla \hat u (s, \hat \psi(s, \YY^m_s) ) \right|^2\di s\right]\notag
\end{align}
having removed the martingale term. Notice that the last term is analogous to $I^{N, \eps}_t$ in \eqref{eq:l2} and with very similar calculations we get 
\begin{align}\label{eq:lt1}
\frac1\eps \E &\left[\int_0^T \mathds{1}_{\{\YY_s\!-\!\YY^m_s>\eps\}}e^{1-(\YY_s\!-\!\YY^m_s)/\eps}  \left| \nabla \hat u (s, \hat \psi(s, \YY_s) )  -  \nabla \hat u (s, \hat \psi(s, \YY^m_s) ) \right|^2\di s\right]\\
& \leq  \frac1\eps\left( \| \hat u\|^2_{\infty,\C^{1,\gamma}}2^{2\gamma} T\eps^{2\gamma\zeta}+T \exp(1-\eps^{\zeta-1})\right) ,\notag
\end{align}
with $1/2<\gamma<\gamma_0$ and any $\zeta\in(0,1)$ such that $2\gamma\zeta>1$. Hence the right-hand side of  \eqref{eq:lt1} tends to zero as $\eps\to 0$. Noting that also $4\eps \to 0$ in \eqref{eq:lt} concludes the proof. 
\end{proof}

We are now ready to prove Proposition \ref{thm:EM}, which we recall below for the reader's convenience.

\noindent{\bf Proposition \ref{thm:EM}.}~{\em 
Let Assumption \ref{ass: main} hold and let $b^N\in \mathcal{C}^{\frac{1}{2}}([0,T];H^0_{\tilde q_0, q_0})$ for some fixed $N$. Then, as $m\to\infty$, we have 
\begin{align}\label{eq:EMb}
\sup_{0\le t \le T}\E\left[|X^{N,m}_t\!-\!X^N_t|\right]\le C_2(N) m^{-1}+C_3(N) m^{-\frac{1}{2}} 
\end{align}
with
\begin{align}\label{eq:C23b}
\begin{split}
&C_2(N):=c\,\|P_{\eta_N} b^N\|_{\infty,L^\infty}\Big(1 +\|\nabla (P_{\eta_N} b^N)\|_{\infty,L^\infty}\Big),\\
&C_3(N):=c'\,\Big(\|\nabla (P_{\eta_N} b^N)\|_{\infty,L^\infty}+[P_{\eta_N} b^N]_{\frac12,L^\infty}\Big)
\end{split}
\end{align}
and $c,c'>0$ constants independent of $(N,m)$.
}
\begin{proof}
Since $ \hat \psi(t,\cdot) =  \psi^N(t,\cdot)$ is 2-Lipschitz we have
\begin{align}\label{eq:bb2}
\E\left[\left|\XX^m_t-\XX_t\right|\right]=\E\left[\left|\hat \psi(t,\YY^m_t)-\hat \psi(t,\YY_t)\right|\right]\le 2 \E\left[\left|\YY^m_t-\YY_t\right|\right].
\end{align}
We apply It\^o-Tanaka formula to $\big|\YY^m_t-\YY_t\big|$, using \eqref{eq:YY1} and \eqref{eq:YY2}, and removing the martingale term by taking expectation. Thus we obtain
\begin{align*}
\E&\left[\left|\YY^m_t-\YY_t\right|\right]\\
=&(1+\lambda)\E\left[\int_0^t\mathrm{sign}(\YY^m_s-\YY_s)\left( \hat u(s,\hat\psi(s,\YY^m_s))- \hat u(s,\hat\psi(s,\YY_s))\right)\di s\right]\notag\\
&+\E\left[\int_0^t \mathrm{sign}(\YY^m_s-\YY_s)\di  E^m_s\right]+\frac{1}{2}\E\left[L^0_t(\YY^m-\YY)\right]\notag\\
\le & 
(1+\lambda)\E\left[\int_0^t\left|\YY^m_s-\YY_s\right|\di s\right]+\E\left[\int_0^t\left|\di E^m_s\right|\right]+\frac{1}{2}\E\left[L^0_t(\YY^m-\YY)\right],\notag
\end{align*}
where we have also used that, $ \psi^N(t,\cdot)$ is 2-Lipschitz, uniformly in $t\in[0,T]$ and $\hat u$ is $\tfrac12$-Lipschitz. Applying Lemma \ref{lem:localtime} to the term featuring the local time we get
\begin{align*}
\E\left[\left|\YY^m_t-\YY_t\right|\right]
\le 3(1+\lambda) \E\left[\int_0^t\left|\YY^m_s-\YY_s\right|\di s\right]+2\E\left[\int_0^t\left|\di E^m_s\right|\right].
\end{align*}
Then, by Lemma \ref{lem:E} we obtain
\begin{align}
\E&\left[\left|\YY^m_t-\YY_t\right|\right]\\
\le & 3 (1+\lambda)
\E\left[\int_0^t\left|\YY^m_s-\YY_s\right|\di s\right]
+2 c_1(N) m^{-1}+2 c_2(N)m^{-\frac{1}{2}}\notag
\end{align}
with $c_1,c_2>0$ as in Lemma \ref{lem:E}. 
By Gronwall's inequality we get
\begin{align}\label{eq:bb3}
\E&\left[\left|\YY^m_t-\YY_t\right|\right]\le \tfrac{1}{2}C_2(N)  m^{-1} + \tfrac{1}{2}C_3(N) m^{-\frac{1}{2}}
\end{align}
where
\begin{align}\label{eq:c2}
C_2(N):=4c_1(N) e^{3(1+\lambda)T}\quad\text{and}\quad C_3(N):=4c_2(N) e^{3(1+\lambda)T}.
\end{align}
Then, finally plugging \eqref{eq:bb3} into \eqref{eq:bb2} we obtain \eqref{eq:EMb}.
\end{proof}

%%%%%%%%%%%%%%%%%%%%%%%%%%%%%%%%%%%%%%%%%%%%%%%%%%%%%%%%%%%%%%%%%%%

\appendix
\section{Haar and Faber basis}\label{sc: haar}
In this appendix we introduce Haar and Faber functions and discuss some of their key properties. These functions form a basis for certain fractional Sobolev spaces, which we use throughout the paper. Roughly speaking, Haar functions are `step functions' that form a basis for  $H^s_r$ with $-1/2<s<1/r$ and $2\leq r <\infty$, while Faber functions are `hat functions' (obtained by integrating Haar functions) that form a basis for $H^s_r$ with $1/2<s<1/r+1$ and  $2\leq r <\infty$. 

Using these bases it is possible to represent an element $f$ from either of those fractional Sobolev spaces in terms of infinite sums. Moreover, the sums can be cut to finite sums as a way of approximating the original function $f$. This procedure can be made rigorous thanks to the theory of fractional Sobolev spaces and to the properties of Haar and Faber function. The link between Haar and Faber representations is useful for numerical implementation and worth commenting on. Distributional drifts as those we consider in this paper can be often obtained as the (distributional) derivative of a function $g$ in $H^s_r$ with $1/2<s<1/r+1$ and  $2\leq r <\infty$. In that case, it is easy to obtain the coefficients on the Haar basis expansion by evaluating $g$ at a finite number of points (see Remark \ref{rmk: haar faber link}). Below we recall the key results and definitions that we use in this paper. 

Throughout the section we denote by $\mathcal S$ the space of Schwartz functions, and by $\mathcal S'$ its dual (the space of Schwartz distributions). Moreover we use $\mathcal D= \C_c^\infty$ to indicate $\C^\infty$-functions  with compact support and $\mathcal D'$ for its dual.
Next we introduce the Haar wavelet system on $\R$, see   \cite[equations  (2.93)--(2.96)]{Tri-bas}.
\begin{definition}[Haar wavelets on $\R$] \label{def: Haar system R}
  Let us define the mother wavelet $  x\mapsto h_M(x)$ by
 $  h_M : =  \mathds{1}_{\left[0,\frac{1}{2}\right)}-\mathds{1}_{\left[\frac{1}{2},1\right)} $. 
  The Haar wavelet  system on $\R$ is given by
\begin{equation}\label{eq: haar system R}
\{h_{j,m}:  j\in\N\cup\{-1\},\ m \in\Z\}
\end{equation}
where $ h_{-1,m}(x):=\sqrt{2}\left|h_M(x-m)\right|$  for $m\in\Z$,  
and $   h_{j,m}(x):= h_M(2^jx-m)$  for $ j\in\N$ and $m\in\Z$. 
Alternatively we can rearrange the system \eqref{eq: haar system R} as follows 
\begin{equation}\label{eq: haar system R bis}
\{h_{j,m}^k:  j\in\N\cup\{-1\},\ k \in\Z,\  m=0, \ldots, 2^j-1\}
\end{equation}
where $h_{-1,0}^k(x):= h_{-1,k}(x)$ for all $k\in\Z$, and $h_{j,m}^k(x):= h_{j,m}(x-k)$ for all $j\in\N, \ k \in\Z $ and $  m=0, \ldots, 2^j-1$.
\end{definition}

See Figure \ref{fig: haar wavelet} for the plot of a generic Haar function $h_{j,m}$. 
For future reference note that
\begin{align}\label{eq:hjm} 
 h_{j,m}(x)= \begin{cases}
 1 & \text{if } x\in [\frac m{2^j}, \frac{m+1/2}{2^j} )\\
 -1 & \text{if } x\in [\frac{ m+1/2}{2^j}, \frac{m+1}{2^j} ) \\
 0 &\text{else.}
\end{cases}
\end{align}

It turns out that the Haar wavelets system \eqref{eq: haar system R} (or equivalently \eqref{eq: haar system R bis}) is an unconditional basis for fractional Sobolev spaces on $\R$ of order ``close to zero'' (i.e.\ ${-1/2}<s<{1/r}$), as detailed in the theorem below which is taken from \cite[Theorem 3.3 and  Remark 3.4]{Tri-fab}.
\begin{theorem}\label{thm: haar R}
Let $0\leq r < \infty$, $ -\frac{1}{2}<{s}<\frac{1}{r} $, and let $f\in \mathcal S'(\R)$. Then $f\in H^{s}_r(\R)$ if and only if it can be represented as  
 \begin{equation}\label{eq: haar represent Rk} 
    f = \sum_{j=-1}^{+\infty}\sum_{k\in\Z} \sum_{m=0}^{2^j-1} \mu_{j,m}^k2^{-j\left({s}-\frac{1}{r}\right)}h_{j,m}^k,
 \end{equation} 
 with unconditional convergence in $\mathcal S'(\R)$  and locally in any space $H^{\sigma}_r(\R)$ with $\sigma<s$. Here  $\sum_{m=0}^{2^j-1}$ means $m=0$ when $j=-1$.  

The representation is unique, with the coefficients   given by
 \begin{equation*}
   \mu_{j,m}^k := 2^{j\left({s}-\frac{1}{r}+1\right)}\int_{\R}f(x)h_{j,m}^k(x) \di x,
\end{equation*} 
 where the   integral is to be understood  in the sense of  dual pairing.  Moreover the system 
\[
\left\{2^{-j\left({s}-\frac{1}{r}\right)}h_{j,m}^k:\ j\in\N\cup\{-1\},\ k\in\Z, \ m=0, \ldots, 2^j-1 \right\}
\]
is an unconditional normalised basis of $H^{s}_r(\R)$.
\end{theorem}
It is shown in \cite[Theorem 2.9 and  Remark 2.12]{Tri-bas} that  \eqref{eq: haar represent Rk} can be  equivalently written as 
\begin{align}\label{eq:muh}
 f = \sum_{j=-1}^{+\infty}\sum_{m\in\Z} \mu_{j,m}2^{-j\left({s}-\frac{1}{r}\right)}h_{j,m},
 \end{align}
where $\mu_{j,m}=\mu^0_{j,m}$ and $h_{j,m}$ are as in \eqref{eq:hjm}.

\begin{remark}\label{rm: haar coeff}
Let us denote 
\begin{equation}\label{eq: haar coeff tilde}
\tilde{\mu}_{j,m}^k  := 2^{j}\int_{\R}f(x)h_{j,m}^k(x)\ \di x = 2^{j\left({-s}+\frac{1}{r}\right)}\mu_{j,m}^k.
\end{equation}
It follows from Theorem \ref{thm: haar R} and \eqref{eq:muh} that $\tilde \mu_{j,m}^0 = \tilde \mu_{j,m}$, where $\tilde\mu_{j,m}$ is defined as in \eqref{eq: haar coeff tilde} but with $h_{j,m}^k $ replaced by $h_{j,m}$. 

Moreover, from \eqref{eq: haar represent Rk} we get the more compact representation
\begin{equation}\label{eq: haar repres k}
f = \sum_{j=-1}^{+\infty}\sum_{k\in\Z}\sum_{m=0}^{2^j-1}\tilde{\mu}_{j,m}^k h_{j,m}^k.
\end{equation}
\end{remark}

\begin{remark}\label{rm: haar coeff pq}
The coefficients  $\tilde{\mu}_{j,m}^k$ do not actually depend on  $r$ or $s$. More precisely,   if $f\in H^s_r(\R)\cap H^{s'}_{r'}(\R)$ for some $r\neq r'$ and $s\neq s'$ then the representation \eqref{eq: haar repres k} is exactly the same in both spaces, with the same coefficients.
\end{remark}

In what follows we analyse fractional Sobolev spaces on an open bounded interval $I\subset\R$.  Let us recall \cite[Definition 1.24 (i)]{Tri-bas}: let $I$ be an open set in $\R$, then  
\begin{align}\label{eq:distr}
 H^s_ {r} (I): = \{ f \in \mathcal D'(I) : f = h_{\vert I} \text{ for some }h\in H^s_ {r}(\R) \}
\end{align}
endowed with the norm 
\[
\|f \|_{ H^s_ {r} (I)} = \inf\{ \|h\|_{ H^s_ {r} (\R)};\,h\in  H^s_ {r} (\R)\:\text{with}\: f = h_{\vert I}  \}.
\] 

With no loss of generality we specialise to $I:=(0,1)$ for simplicity of exposition. Next we introduce the Haar wavelet system on $I$ (see \cite[equations (2.128) and (2.129)]{Tri-bas} for details) which is useful for the fractional Sobolev space $H^s_ {r}(I)$. 
\begin{definition}[Haar wavelets on $I=(0,1)$] \label{def: Haar system I}
The Haar wavelet  system on $I$ is given by
\begin{equation}\label{eq: haar system I}
\{h_0, h_{j,m}:  j\in\N ,\ m =0, \ldots, 2^j-1\}
\end{equation}
where
\[ 
h_{0}(x):= \mathds{1}_I (x) 
\]
and $h_{j,m}$ is as in \eqref{eq:hjm}.
\end{definition}
Notice that the system \eqref{eq: haar system I} is essentially the restriction of \eqref{eq: haar system R} to the interval $I$. In particular, $h_{j,m}$ is now restricted to values of $m$ between $0$ and $2^j-1$ rather than $m\in\mathbb{Z}$ as in Definition \ref{def: Haar system R}. Moreover, the set of elements $h_{-1, m}$, defined on $\R$, with $m\in\Z$, has been replaced by $h_0$, defined on $I$. For the fractional Sobolev spaces on $H^s_ {r}(I)$ we have again a representation in terms of Haar functions, as illustrated below. For more details see \cite[Theorem 2.13]{Tri-bas}.

\begin{theorem}\label{thm: haar I}
Let $2\leq  {r} < \infty$, $ -\frac{1}{2}<{s}<\frac{1} {r} $, and let $f\in \mathcal D'(I)$. Then  $f\in H^{s}_ {r}(I) $ if and only if it can be represented as  
 \begin{equation}\label{eq: haar represent Ia}
    f = \mu_0 h_0 + \sum_{j=0}^{+\infty}\sum_{m=0}^{2^j-1}\mu_{j,m}2^{-j\left({s}-\frac{1} {r}\right)}h_{j,m} ,
 \end{equation} 
with unconditional convergence in  any space $H^{\sigma}_ {r}(I)$ with $\sigma<s$.
The representation is unique, with the coefficients   given by 
\begin{equation}\label{eq: mu0}
\mu_0  := \int_{I}f(x)h_{0}(x) \di x
\end{equation} 
and, for $j\in\N$ and $m =0, \ldots, 2^j-1 $, by
 \begin{equation*} 
      \mu_{j,m} := 2^{j\left({s}-\frac{1} {r}+1\right)}\int_{I}f(x)h_{j,m}(x) \di x,
    \end{equation*} 
 where the   integrals are to be understood  in the sense of  dual pairing. Moreover the system 
\[
\left\{ h_0, 2^{-j\left({s}-\frac{1} {r}\right)}h_{j,m}:\ j\in\N ,\ m =0, \ldots, 2^j-1 \right\}
\]
is an unconditional normalised basis of  $H^{s}_{r}(I)$.
\end{theorem}

Notice that \eqref{eq: haar represent Ia} can be written in terms of $\tilde   \mu_{j,m}$ (see Remark \ref{rm: haar coeff}) as
  \begin{equation}\label{eq: haar represent I}
  f = \mu_0 h_0 + \sum_{j=0}^{+\infty}\sum_{m=0}^{2^j-1}\tilde \mu_{j,m} h_{j,m} .
\end{equation}
Of course a distribution $f$ defined on $I$ can be seen as a distribution defined on $\mathbb{R}$ but only supported on $I$ (in the sense that $f(\phi)=0$ for all $\phi\in \mathcal D$ supported on $\R\setminus I$). The link between the series representations on $I$ and on $\R$ is given in the next lemma.
\begin{lemma}\label{lm: haar I R}
If $f\in H^s_{r}(\R)$ and supp$(f)\subset I$ then its representation on $\R$ given by \eqref{eq: haar repres k} (or equivalently by \eqref{eq: haar represent Rk}) coincides with its representation on $I$ given by  \eqref{eq: haar represent I}.
\end{lemma}
\begin{proof}
First we remark that in this case the restriction of $f$ to $I$ (denoted again by $f$) belongs to $  H^s_ {r}(I)$ by definition of the latter space. 
Since $\text{supp} (f)\subset I$,  it   follows that $\tilde \mu_{j,m}^k =0 $ for all $k\neq 0$ because the functions $h_{j,m}^k$ are supported on $(k, k+1)$  while $f$ is supported on $(0,1)$ which implies that the dual pairing between $h_{j,m}^k$ and $f$ is non-zero only if $k=0$. Hence the Haar representation \eqref{eq: haar repres k}   becomes 
\begin{align*} \nonumber
f&= \sum_{j=-1}^{+\infty} \sum_{k=0} \sum_{m=0}^{2^j-1} \tilde   \mu_{j,m}^k  h_{j,m}^k  \\ \nonumber
&= \sum_{j=-1}^{+\infty}  \sum_{m=0}^{2^j-1} \tilde   \mu_{j,m}   h_{j,m}\\
&= \tilde  \mu_{-1,0}   h_{-1
,0}   +  \sum_{j=0}^{+\infty}  \sum_{m=0}^{2^j-1} \tilde   \mu_{j,m}   h_{j,m},
\end{align*}
where we used the fact that $\tilde \mu_{j,m}^0 =\tilde \mu_{j,m}  $. The proof can be concluded by noticing that 
\begin{align*}
\tilde\mu_{-1,0} h_{-1,0} &=  2^{-1} \int_\R f(x) h_{-1,0}(x) \di x   h_{-1,0}  \\
&= 2^{-1} \int_\R f(x) \sqrt 2 \mathds 1_{(0,1)}(x) \di x \sqrt 2  \mathds 1_{(0,1)}\\
 &= \int_I f(x) \mathds 1_{(0,1)} (x)\di x \mathds 1_{(0,1)} = \mu_{0}   h_{0} . \qedhere
\end{align*}
\end{proof}

Next we recall the definition of Faber functions. They are denoted by $v_{jm}$ and are hat-functions, defined as the normalised integrals of the Haar functions $h_{jm}$ on $I$.  Notice that the Faber series representation holds in general only on bounded domains in $\R$. Here we only recall their definition on the unit interval $I=(0,1)$. More details can be found in \cite[Section 3.2.1]{Tri-fab}.

\begin{definition}[Faber basis on $I$]\label{def: Faber system I}
The Faber system on $(0,1)$ is given by 
\begin{equation*}
\left \{ v_0, v_1, v_{j,m} : j\in \N  , m=0, \ldots, 2^j-1 \right \}
\end{equation*}
where
\begin{align*}
& v_{0} (x) : = 1-x\qquad\text{and}\qquad v_{1} ( x): =  x
\end{align*}
for $0\leq x\leq 1$ (and zero outside $I$), and the hat-functions are defined as 
\[
v_{j,m}(x) = 2^{j+1}\int_0^x h_{j,m}(y) \di y,
\]
that is
\[ 
 v_{j,m}(x):= \begin{cases}
 2^{j+1}(x-2^{-j}m) & \text{if } x\in [\frac m{2^j}, \frac{m+1/2}{2^j} )\\
 2^{j+1}(2^{-j}(m+1) - x) & \text{if } x\in [\frac{ m+1/2}{2^j}, \frac{m+1}{2^j} ) \\
 0 &\text{else.}
\end{cases}
\]  
\end{definition} 
Using the Faber system on $I$ it is possible to represent elements of fractional Sobolev spaces on domain $I$ for $1/2 < s< 1+1/ {r}$ and $2\leq  {r}<\infty$ as we see below. For a proof see \cite[Theorem 3.1 and Corollary 3.3]{Tri-bas}. 
\begin{theorem}\label{thm: faber I}
    Let $g\in H^{s}_ {r}(I)$ for $2\leq  {r} < \infty$, and $\frac{1}{2}<s< 1+\frac{1} {r}$. Then we have the unique Faber representation for $g$ 
    \begin{align*}
    g &= \bar \mu_0v_0 +\bar \mu_1v_1 +\sum_{j=0}^{+\infty}\sum_{m=0}^{2^j-1}\bar \mu_{j,m}  v_{j,m}
      \end{align*}
    with unconditional convergence in $\mathcal{C}(I)$ and in $H^\sigma_ {r} (I)$ with $\sigma<s$.
     Here the coefficients $\bar \mu$ are explicitly given by 
    \begin{equation}\label{eq: faber coeff}
    \begin{cases}
   \bar  \mu_{j,m} &=-\frac12 \left(\Delta^2_{2^{-j-1}}g\right)(2^{-j}m)\\
     \bar   \mu_{0} &= g(0)\\
     \bar   \mu_{1} &= g(1)\\
    \end{cases}
    \end{equation}
and where   $(\Delta^2_{h}g)(x) := g(x+2h)-{2}g(x+h)+ g(x)$.
\end{theorem}
This representation of $g$ using Faber functions is fundamental to calculate the coefficients for the Haar representation of $g'$, as we see below. 
\begin{remark}\label{rmk: haar faber link}
In the proof of the above theorem (see \cite[Theorem~3.1, Corollary~3.3]{Tri-bas}) the  following expansion for $g'\in H^{s-1}_ {r}(I)$ is derived 
\[
g'= (\bar \mu_1 - \bar \mu_0)h_0+ \sum_{j=0}^{+\infty}\sum_{m=0}^{2^j-1}2^{j+1}\bar \mu_{j,m}h_{j,m}.
\]
Comparing it with \eqref{eq: haar represent I} from  Theorem \ref{thm: haar I}, with $f=g'$,  we obtain an explicit representation of the coefficients $\tilde \mu_{j,m}$ and $\mu_0$ that appear in \eqref{eq: haar represent I}, that is
\begin{align}\label{eq: haar faber link}
& \mu_0 = \bar \mu_1 - \bar \mu_0 \qquad\text{and}\qquad \tilde \mu_{j,m} = 2^{j+1}\bar \mu_{j,m}. 
\end{align}
\end{remark}
The link expressed in \eqref{eq: haar faber link} together with the explicit expression \eqref{eq: faber coeff} is crucial to evaluate numerically the coefficients in the Haar expansion of an element $f\in H^{s-1}_ {r}(I)$ for $\tfrac12<s<1+\tfrac1r$ and $2\leq r <\infty$. Indeed to do so we only need to evaluate the associated function $g$ at (a finite number of) mesh points
\begin{equation}\label{eq: explicit repr}
\begin{cases}
\mu_0 &= g(1) - g(0)\\ 
\tilde \mu_{j,m} &= -2^j \left( g(\frac{m+1}{2^j}) -2  g(\frac{m+1/2}{2^j}) +  g(\frac{m}{2^j})\right).
\end{cases}
\end{equation} 

\subsection{Equivalent norms and coefficients of Haar series}\label{sc:coeffs}
For each $t\in[0,T]$ we have, from \cite[Theorem 2.9 and eq.\ (2.114), Sec.\ 2.2.3]{Tri-bas}, that the norms $\|b(t)\|_{ H^{-\beta}_{q}}$ and $\|\boldsymbol{\mu}(t;\beta,q)|f^-_{q,2}\|$ are equivalent, where
\begin{align}\label{eq:tribnorm}
\|\boldsymbol{\mu}(t;\beta,q)|f^-_{q,2}\|:=&\Big[\int_{\mathbb R}\Big( \sum_{j=-1}^\infty \sum_{m\in\Z}\big|\mu_{j,m}(t;\beta,q)2^{\frac{j}{q}}\mathds{1}_{[2^{-j}m,2^{-j}(m+1))}(x)\big|^2\Big)^{\frac{q}{2}}\di x\Big]^\frac{1}{q}.
\end{align}
We note that if $q=2$ the expression above simplifies to 
\begin{align}\label{eq:tribnorm2}
\|\boldsymbol{\mu}(t;\beta,2)|f^-_{2,2}\|=&\Big[\int_{\mathbb R}\Big( \sum_{j=-1}^\infty \sum_{m\in\Z}\big|\mu_{j,m}(t;\beta,2)\big|^2 2^{j}\mathds{1}_{[2^{-j}m,2^{-j}(m+1))}(x)\Big)\di x\Big]^\frac{1}{2}\\
=&\Big[ \sum_{j=-1}^\infty \sum_{m\in\Z}\big|\mu_{j,m}(t;\beta,2)\big|^2 2^{j}\int_{\mathbb R}\mathds{1}_{[2^{-j}m,2^{-j}(m+1))}(x)\di x\Big]^\frac{1}{2}\notag\\
=&\Big[ \sum_{j=-1}^\infty \sum_{m\in\Z}\big|\mu_{j,m}(t;\beta,2)\big|^2\Big]^{\frac{1}{2}},\notag
\end{align}
where we can swap the sums and the integral by monotone convergence.
Since the norms are equivalent for all $t\in[0,T]$ then 
\begin{align}\label{eq:equivn}
\|b\|_{\infty,H^{-\beta}_q}<\infty\iff\sup_{t\in[0,T]}\|\boldsymbol{\mu}(t;\beta,q)|f^-_{q,2}\|<\infty.
\end{align}

\section{Some estimates for the (killed) heat semigroup}\label{sc:P}

{\em Outline of the derivation of Eq.\ \eqref{eq:IZ1}}. Let us start from the first equation in \eqref{eq:IZ1}. To show that $P_t$ is a contraction on $H^{-s}_r$ for $s>0$ we write
\[
\|P_t f \|_{H^{-s}_r} =  \|A^{-s/2}P_t A^{s/2}A^{-s/2}f \|_{L^r} =  \|P_t A^{-s/2}f \|_{L^r},
\] 
where in the final equality we use that for $s>0$ the operators $A^{s/2}$ and $P_t$ commute by \cite[Theorem II.6.13]{pazy83}.
Since $A^{-s/2}f\in L^r(\R)$ and $P_t$ is a contraction on $L^r(\R)$ we get $\|P_t A^{-s/2} f \|_{L^r} \leq \|A^{-s/2} f \|_{L^r} = \|f\|_{H^{-s}_r}$, as needed. To prove that $P_t$ is a contraction for $H^{s}_r$ the idea is the same but we write $P_t f = A^{-s/2}A^{s/2} P_t f$.  

Let us now prove the second equation in \eqref{eq:IZ1}. First we recall that $\|P_t g -g \|_{L^r} \leq c\, t^{\eps/2} \|A^{\eps/2} g\|_{L^r}$ if $g\in D(A^{\eps/2})$, from \cite[Theorem II.6.13]{pazy83}.  For $f\in H^{-s}_r$ we choose $g=  A^{-(s+\eps)/2} f$ and we get 
\[
\|P_t f -f\|_{H^{-(s+\eps)}_r}  = \|A^{-(s+\eps)/2}(P_t f -f)\|_{L^r} =  \|(P_t A^{-(s+\eps)/2} f - A^{-(s+\eps)/2}f)\|_{L^r},
\] 
where in the final equality we use again that $A^{s/2}$ and $P_t$ commute.
Since $A^{-s/2}f \in L^r$ then $g=A^{-(s+\eps)/2}f \in H^\eps_r = D(A^{\eps/2})$ so we get
\[
\|(P_t A^{-(s+\eps)/2} f - A^{-(s+\eps)/2}f)\|_{L^r}  \leq c\, t^{\eps/2} \|A^{-s/2}f \|_{L^r} = c\, t^{\eps/2} \|f \|_{H^{-s}_r},
\]
as needed.
\medskip

{\em Outline of the derivation of Eq.\ \eqref{eq:IZ2}}. Let us start from the first equation in \eqref{eq:IZ2}. By \cite[Theorem II.6.13]{pazy83}, for $t\in(0,T)$, $\nu\geq0$ and $g\in L^r$ we have that $\|A^{\alpha/2 } P_t g\|_{L^r} \leq c \, t^{-\alpha/2} \|g\|_{L^r}$. Then choosing $\alpha = \nu+s$, $g = A^{-s/2} f$, and using again that $P_t$ and $A^{s/2}$ commute, we get  
\begin{align*}
\|P_t f\|_{H^\nu_r}  =& \| A^{\nu/2} P_t f\|_{L^r} = \| A^{(\nu+s)/2} P_t  A^{-s/2} f\|_{L^r}\\
 \leq& c\, t^{-(\nu+s)/2} \|A^{-s/2} f\|_{L^r} = c\,  t^{-(\nu+s)/2} \| f\|_{H^{-s}_r}.
\end{align*} 
Now if $\nu>1/r$ we can use fractional Morrey inequality and obtain 
\[
\|P_t f \|_{L^\infty} \leq \|P_t f \|_{C^{0,\alpha}} \leq \|P_t f\|_{H^\nu_r},
\]
by which we conclude. Similarly for the second bound in \eqref{eq:IZ2}, simply replace $\nu$ by $1+\nu$ and use  $\|\nabla P_t f \|_{L^\infty} \leq \|P_t f \|_{C^{1,\alpha}} \leq \|P_t f\|_{H^{1+\nu}_r}$ to  conclude.

\bibliographystyle{plain} 	
\bibliography{pre}

\end{document}